\documentclass[usenames,dvipsnames]{amsart}
\usepackage{graphics}
\usepackage{graphicx}
\usepackage{appendix}
\usepackage{amsfonts,amssymb,amsopn,amscd,amsmath}
\usepackage{mathrsfs}
\usepackage{pgfplots}\pgfplotsset{compat=1.13}
\usetikzlibrary{arrows,positioning,calc,patterns}

\usepackage[pagebackref,colorlinks,urlcolor=black,hypertexnames=true]{hyperref}
 \newtheorem{thm}{Theorem}[section]
 \newtheorem{cor}[thm]{Corollary}
 \newtheorem{lem}[thm]{Lemma}{\rm}

 {\rm}
 \newtheorem{assumption}[thm]{Assumption}
 \newtheorem{rem}[thm]{Remark}
 
\numberwithin{equation}{section}
\DeclareMathOperator{\vol}{vol}
\usepackage{tikz}
\usetikzlibrary{patterns,decorations.pathreplacing}
\begin{document}

\def\red{\color{red}}
\def\bl{\color{blue}}
\def\ora{\color{orange}}
\def\green{\color{green}}
\def\br{\color{brown}}

\def\la{\langle}
\def\ra{\rangle}
\def\e{{\rm e}}
\def\x{\mathbf{x}}
\def\by{\mathbf{y}}
\def\bz{\mathbf{z}}
\def\F{\mathcal{F}}
\def\R{\mathbb{R}}
\def\T{\mathbf{T}}
\def\N{\mathbb{N}}
\def\K{\mathbf{K}}
\def\bK{\overline{\mathbf{K}}}
\def\Q{\mathbf{Q}}
\def\M{\mathbf{M}}
\def\O{\mathbf{O}}
\def\C{\mathbf{C}}
\def\P{\mathbf{P}}
\def\Z{\mathbb{Z}}
\def\H{\mathcal{H}}
\def\A{\mathbf{A}}
\def\V{\mathbf{V}}
\def\AA{\overline{\mathbf{A}}}
\def\B{\mathbf{B}}
\def\c{\mathbf{C}}
\def\L{\mathbf{L}}
\def\bS{\mathbf{S}}
\def\H{\mathcal{H}}
\def\I{\mathbf{I}}
\def\Y{\mathbf{Y}}
\def\X{\mathbf{X}}
\def\G{\mathbf{G}}
\def\B{\mathbf{B}}
\def\f{\mathbf{f}}
\def\z{\mathbf{z}}
\def\y{\mathbf{y}}
\def\d{\hat{d}}
\def\bx{\mathbf{x}}
\def\y{\mathbf{y}}
\def\v{\mathbf{v}}
\def\g{\mathbf{g}}
\def\w{\mathbf{w}}
\def\b{\mathcal{B}}
\def\a{\mathbf{a}}
\def\q{\mathbf{q}}
\def\u{\mathbf{u}}
\def\s{\mathcal{S}}
\def\cc{\mathcal{C}}
\def\co{{\rm co}\,}
\def\cp{{\rm CP}}
\def\tg{\tilde{f}}
\def\tx{\tilde{\x}}
\def\supmu{{\rm supp}\,\mu}
\def\supnu{{\rm supp}\,\nu}
\def\m{\mathcal{M}}
\def\bR{\mathbf{R}}
\def\om{\mathbf{\Omega}}
\def\c{\mathbf{c}}
\def\s{\mathcal{S}}
\def\bs{\mathbf{s}}
\def\k{\mathcal{K}}
\def\la{\langle}
\def\ra{\rangle}
\def\blambda{{\boldmath{\lambda}}}
\def\bsmlambda{\boldmath{\lambda}}
\def\ov{\overline{o}}
\def\und{\underline{o}}
\title[Computing the Hausdorff boundary measure of semi-algebraic sets]{Computing the Hausdorff boundary measure of semi-algebraic sets}
\thanks{This work has been supported by European Union’s Horizon 2020 research and innovation programme under the Marie Sklodowska-Curie Actions, grant agreement 813211 (POEMA). 
The research of the first author was funded by the European Research Council (ERC) under the European’s Union Horizon 2020 research and innovation program (grant agreement 666981 TAMING. 
The second author was supported by the FMJH Program PGMO (EPICS project) and  EDF, Thales, Orange et Criteo, from the Tremplin ERC Stg Grant ANR-18-ERC2-0004-01 (T-COPS project),  as well as ANITI, coordinated by the Federal University of Toulouse within the framework of French Program ``Investing for the Future C PIA3'' program under the Grant agreement n$^{\circ}$ANR-19-XXXX-000X}

\author{Jean-Bernard Lasserre and Victor Magron}

\address{Jean-Bernard Lasserre: LAAS-CNRS and Institute of Mathematics\\
University of Toulouse\\
LAAS, BP 54200, 7 avenue du Colonel Roche\\
31031 Toulouse C\'edex 4,France}
\email{lasserre@laas.fr}
\address{Victor Magron: LAAS-CNRS, BP 54200, 7 avenue du Colonel Roche\\
31031 Toulouse C\'edex 4,France}
\email{vmagron@laas.fr}
\date{}

\begin{abstract}
Given a compact basic semi-algebraic set $\om\subset\R^n$ we provide a numerical scheme
to approximate as closely as desired, any finite number of 
moments of the Hausdorff measure $\sigma$ on the boundary $\partial\om$. This also allows one to 
approximate interesting quantities like ``length", ``surface", or more general integrals on the boundary,
as closely as desired from above and below.
\end{abstract}



\maketitle

\section{Introduction}

This paper is concerned with the Hausdorff boundary measure of a compact \emph{basic semi-algebraic set} $\om$, 
that is, a compact set defined by a finite conjunction of polynomial inequalities. Our main  contribution is to provide a 
systematic numerical scheme to approximate as closely as desired any (fixed) finite number of its moments, in particular its mass (the length or area of $\partial\om$).

Besides being a challenge of its own in computational mathematics, computation of moments (e.g. the mass) of the boundary measure has also important practical applications: 

- either practical ones, e.g., in computational geometry (perimeter, surface area), or in control for computing the length of trajectories in the context of robot motion planning~\cite{canny1988complexity},

- or theoretical ones such as (real) periods computation~\cite{kontsevich2001periods}.
Periods are integrals of rational functions with rational coefficients over semi-algebraic sets and the moments of the Hausdorff boundary measure of $\om$ are special cases of periods where the integrated rational functions are monomials.

In some applications (e.g. tomography), the moments of the Lebesgue measure on $\om$ are available after appropriate measurements. In this case the methodology slightly simplifies as they now appear as data instead of variables.

Of course certain \emph{line} or \emph{surface} specific integrals on $\partial\om$ reduce to surface of volume
integrals of a related function on $\om$ via Green's (or Stokes') theorem, in which case one may invoke the arsenal of techniques of multivariate integration on specific domains $\om$
(like Monte-Carlo and/or cubatures techniques). 
However, we have not been able to find in the literature a systematic numerical scheme for computing ``volume" of (or integrals on)
the boundary $\partial\om$ of a basic semi-algebraic set $\om$. To the best of our knowledge this work seems to be a first such attempt, at least at this level of generality.

\subsection*{Background} Among existing techniques for numerical volume computation and integration, Monte Carlo algorithms generate points uniformly in a box which contains $\om$ and approximate the volume by the ratio of the number of points that fall into $\om$, and similarly for integration. But of course such a technique requires $\om$ to be full dimensional.
Cubature formulas perform numerical integration on simple sets, such as simplices, boxes or balls. 
However, such formula are not available for arbitrary semi-algebraic sets. Let us also mention the more recent algorithm \cite{Lairez19} which provides arbitrary precision approximations of the volume of $\om$.
This algorithm is based on computing the Picard-Fuchs differential equations of appropriate periods and critical point properties. 

In contrast, our methodology follows the line of research initiated in~\cite{Las01sos} for solving the \emph{Generalized Moment Problem} (GMP) with algebraic data, by the {\em Moment-SOS Hierarchy} \cite{lasserre-icm}. In this approach the problem on hand (here volume computation or integration) is first modelled as an infinite-dimensional 
Linear Program (LP) on appropriate spaces of Borel measures. Then one approximates the solution of this LP by solving a hierarchy of semidefinite programming (SDP) problems. For a more general overview, we refer the interested reader to \cite{lasserre2009moments,lasserre-icm}. 

It turns out that the GMP is a rich model which encompasses a lot of important applications in various area of science and engineering, some of them described in \cite{lasserre2009moments} and \cite{lasserre-icm}. For instance this approach has been applied in the context of \emph{dynamical} systems with 
polynomial vector field to characterize (i) the backward reachable set~\cite{HK14roa}, (ii) the forward reachable set for discrete-time polynomial systems~\cite{reach19}, and (iii) the maximal positively invariant set for continuous-time polynomial systems~\cite{oustry2019inner}.
Also in~\cite{invsdp19}, the authors describe how to approximate numerically moments and supports of measures which are invariant with respect to the dynamics of polynomial systems, under semi-algebraic set constraints on the trajectories.

For volume computation of a semi-algebraic set $\om$, the GMP formulation as an infinite-dimensional LP has been developed in \cite{HLS09vol}.
It requires knowledge of a simple set $\B\supset\om$ such that all moments of the Lebesgue measure on $\B$ are available
(e.g., $\B$ is a Euclidean ball, an ellipsoid, a box). When $\om$ is the level set of a single homogeneous polynomial, 
the first author recently proposed in~\cite{lasserre2019volume} a related and alternative method 
which results in solving  a hierarchy of generalized eigenvalue problems (rather than a sequence of SDP problems) with respect to a pair of Hankel matrices of increasing size.

\subsection*{Contribution} Our contribution is in the spirit the work  \cite{HLS09vol} as we
also extensively use the GMP formulation as infinite-dimensional LPs on appropriate spaces of measures. Indeed, 
to compute the moments of $\sigma$ on $\partial\om$ we relate them (linearly) to the moments of the Lebesgue measure $\lambda$ on $\om$
via Stokes' theorem. (In \cite{HLS09vol} the Moment-SOS approach was used to approximate moments of $\lambda$.)
Namely:

$\bullet$ In Section~\ref{sec:homogeneous}, we first provide a numerical scheme to approximate as closely as desired any finite number of moments of the Hausdorff measure $\sigma$ on the boundary $\partial\om$ of a basic compact semi-algebraic set $\om$, assuming that $\om$ is described by homogeneous polynomials. This numerical scheme is based on a hierarchy of SDP relaxations whose optimal solutions yield better and better approximations of the moments of $\sigma$. 

Interestingly, in the proposed numerical scheme, moments of the Lebesgue measure on $\om$ appear as data 
or unknowns, depending on whether or not they are already available (e.g. analytically or from measurements).
Importantly, in this case one obtains two sequences of upper and lower bounds on $\sigma(\partial\om)$, with both sequences converging to $\sigma(\partial\om)$. 

$\bullet$ Then in Section~\ref{sec:convex} we extend  the approach
to the case where $\om$ is convex and not necessarily described by homogeneous polynomials. Finally in Section~\ref{sec:general} we also extend the approach to the case where $\om$ is not necessarily convex.
Some numerical experiments are provided in Section~\ref{sec:benchs}.

\subsection*{Underlying technique} One proceeds in two steps: 
First, by Stokes' theorem one relates (via {\em linear constraints}) all moments of the Lebesgue measure $\lambda$ on $\om$ with moments of a measure $\phi$ on $\partial\om$, absolutely continuous with respect to (w.r.t.) 
$\sigma$.
We are thus able to define an infinite-dimensional LP 
with unknowns $\lambda$ and $\phi$, an instance of the GPM with algebraic data. We then apply the Moment-SOS hierarchy
\cite{lasserre-icm} to approximate as closely as desired any fixed number of  moments of both $\lambda$ and $\phi$. If moments of $\lambda$ are already available (say e.g. from measurements as in tomography applications) then they appear as \emph{data} in the SDP-relaxations of the infinite-dimensional LP;
otherwise they are also treated as unknowns of SDP-relaxations.

In a second step we exploit explicit knowledge of the density of $\phi$ w.r.t.  $\sigma$ to relate $\phi$ and 
$\sigma$, again via linear constraints on their moments. This allows  us to define another infinite-dimensional LP
with now (i) the measure $\sigma$ as unknown,  and (ii) $\phi$ (more precisely its moments) already obtained at the first step, as \emph{data}.
This LP is also an instance of the GPM with algebraic data. Then once again we use the Moment-SOS hierarchy to approximate as closely as desired any fixed number of moments of $\sigma$.
It is worth noting that the density of $\phi$ w.r.t. $\sigma$ is {\em not} a polynomial 
as it is of the form $\frac{1}{\Vert \nabla g\Vert}$ for some polynomial $g$. However one is able to handle this semi-algebraic function via an appropriate lifting, a nice useful feature of the Moment-SOS methodology.

\section{Preliminary Background}

\subsection*{Semi-algebraic sets, measures and moments}
With $d,n \in \N$, let $\R[\x]$ (resp. $\R_{d}[\x]$) be the vector space of real-valued $n$-variate polynomials (resp. of degree at most $d$) in the variable $\x=(x_1,\ldots,x_n) \in \R^n$.
Let $\om\subset\R^n$ be the basic compact semi-algebraic set
\begin{equation}
\label{set-om}
\om\,:=\,\{\x\in\R^n:\: g_j(\x)\,\leq\,b_j,\quad j=1,\ldots,m\, \} \,,
\end{equation}
with $g_j \in \R[\x]$, for each $j=1,\dots,m$.
For later purpose, we also note $g_0(x) := 1$ and $b_0 := 0$.
%

Given a compact set $\A \subset \R^n$, we denote by $\mathscr{M}(\A)$ the vector space of finite signed Borel measures supported on $\A$,
and by $\mathscr{M}_+(\A)$ its subset nonnegative elements 
(i.e., positive Borel measures on $\A$).
The support of a measure $\mu \in \mathscr{M}_+(\A)$ is the smallest closed set $C\subset\A$ such that $\mu(\A\setminus C)=0$.

For $\mu,\nu \in \mathscr{M}_+(\A)$, one says that $\mu$ is \emph{dominated} by $\nu$ if $\nu-\mu \in \mathscr{M}_+(\A)$ and we refer to this by the notation $\mu \leq \nu$.
A measure $\mu\in\mathscr{M}_+(\A)$ is said to be \emph{absolutely continuous} with respect to $\nu$ if  for every  $C \in \mathcal{B}(\A)$, $\nu (C) = 0$ implies $\mu(C) = 0$.
In particular, if $\mu \leq \nu$ then $\mu$ is absolutely continuous with respect to $\nu$.\\

Throughout the paper we suppose that the following condition is fulfilled:
\begin{assumption}
\label{hyp:momb}
$\om\subset\B:=(-1,1)^n$.
\end{assumption}
Indeed as $\om$ is compact then Assumption \ref{hyp:momb} is satisfied, possibly after rescaling of the data.
From now on, let $\lambda_\B$ be the Lebesgue measure on $\B$. Its moments 
are denoted $\y^\B=(y^\B_\alpha)_{\alpha\in\N^n}$, i.e.,
\begin{equation}
\label{momb}
y^{\B}_{\alpha} := \int_\B \x^{\alpha} \lambda_\B(d\x)\,=\,\int_\B x_1^{\alpha_1}\cdots x_n^{\alpha_n}\, \lambda_\B (d\x)\,\quad \alpha \in \N^n \,,
\end{equation}
and  are available analytically. In particular $\vert y_\alpha^\B\vert\leq y_0^\B=2^n$, for all $\alpha\in\N^n$.
Other choices of $\B$ (e.g. an Euclidean ball, a box, an ellipsoid) are possible
as soon as all moments of $\lambda_\B$ can be obtained easily or in closed form.

\subsection*{Riesz functional} For a real sequence $\y =(y_{\alpha})_{\alpha \in \N^n} \in \R^{\N^n}$, we define the Riesz linear functional $L_\y : \R[\x] \to \R$ by 
$p\mapsto L_\y(p) := \sum_{\alpha} p_{\alpha} y_{\alpha}$, for all $p \in \R[\x]$. In particular if a sequence $\y$ has a representing measure $\mu$ then $L_\y(p) = \sum_{\alpha} p_{\alpha} y_{\alpha} = \sum_{\alpha} p_{\alpha} \int_\B  \x^{\alpha} \, d \mu = \int_\B p \, d \mu $. 

A measure $\mu$ with moments $(\mu_\alpha)_{\alpha\in\N^n}$ is moment determinate if there is no other measure with 
same infinite sequence of moments.

\subsection*{Moment and localizing matrix} Given $g_j\in\R[\x]$, let $d_j := \deg g_j$ and $r_j = \lceil d_j/2\rceil$, for all $j=0,1,\dots,m$. 
Then for every $j=0,\ldots,m$, the \emph{localizing matrix} $\M_r(g_j \, \y)$ associated with a sequence $\y$ and $g_j$, 
is a real symmetric matrix with rows and columns indexed by $\N_{d-r_j}^{n}$ and defined by: 
\[ 
(\M_d(g_j \, \y))_{\alpha, \beta} := L_\y(g_j(\x) \, \x^{\alpha + \beta}) \,, \quad
\forall \alpha, \beta \in \N_{d-r_j}^n \,. 
\]
For $j=0$ the localizing matrix is called the \emph{moment matrix} $\M_d(\y):=\M_{d-r_0}(g_0\,\y)$.
It is straightforward to check that if $\y$ has a representing measure $\mu\in\mathscr{M}_+(\om)$ then 
$\M_{d-r_j}((b_j-g_j) \, \y) \succeq 0$, for all $j=0,\dots,m$ (the notation $\succeq 0$ stands for positive semidefinite). 
We next provide some useful sufficient conditions for 
converse results.

\subsection*{Sufficient moment conditions}

\begin{thm}
\label{th:bounded-1}
Let $\y=(y_\alpha)_{\alpha\in\N^n}$ $g\in\R[\x]$ be such that 
$\M_d(\y)\succeq0$ and $\M_d(g\,\y) \succeq 0$ for all $j = 1,\dots,m$, and all $d\in\N$. If $\vert y_\alpha\vert\leq c\,M^{\vert\alpha\vert}$ for all $\alpha\in\N^n$ and some $c,M>0$ 
then $\y$ has a unique representing measure $\phi$ on $[-M,M]^n$ and ${\rm supp}(\phi)\subset[-M,M]^n\cap\{\x: g(\x)\geq0\}$.
\end{thm}
\begin{proof}
By \cite[Proposition 3.5(b)]{lasserre2009moments}, $\y$ has a representing measure $\phi$ on $[–M,M]^n$.
Next as $\phi$ has compact support and $\M_d(g\,\y)\succeq0$ for all $d\in\N$, then by \cite[Theorem 3.2(a)]{Las11}
$g(\x)\geq0$ for all $\x\in{\rm supp}(\phi)$.
\end{proof}
\begin{lem}(Multivariate Carleman condition)
\label{carleman}
Let $\y=(y_\alpha)_{\alpha\in\N^n}$ be such that $\M_d(\y)\succeq0$ for all $d\in\N^n$, and
\begin{equation}
    \label{carleman-1}
    \sum_{k=1}^\infty L_{\y}(x_i^{2k})^{-1/2k}\,=\,+\infty,\quad \forall i=1,\ldots,n.\end{equation}
Then $\y$ has a representing measure on $\R^n$ which is moment determinate.
\end{lem}
\subsection*{Stokes' theorem}
Let $\om$ be a smooth manifold with boundary $\partial\om$.
Given a polynomial $p \in \R[\x]$, Stokes' theorem with vector field $X$ states that:
\begin{equation}
\label{stokes-0}
\int_\om {\rm Div}(X p(\x))\,d\x\,=\,\int_{\partial\om}\langle X,\vec{n}_\x\rangle \, p(\x) \,d\sigma \,,
\end{equation}
where $\partial\om$ stands for the boundary of $\om$, $\sigma$ is the $(n-1)$-dimensional Hausdorff boundary measure on $\partial\om$, and $\vec{n}_\x$ is the  outward pointing normal to $\partial\om$; see e.g. Taylor 
\cite[Proposition 3.2, p. 128]{taylor}.
Then Whitney \cite[Theorem 14A]{whitney} generalized Stokes' theorem  to rough domains $\om$ (e.g. with corners). For instance, in our cas:
\begin{equation}
    \label{eq:boundary}
    \partial\om \,=\,\cup_{j=1}^m \om_j\mbox{ with }\om_j\,=\,\{\x\in\om: g_j(\x)\,=\,b_j\,\}.
    \end{equation}
Throughout the paper we assume that
 \begin{equation}
 \label{nondegeneracy}
 \forall\x\in\om_j:\quad \Vert\nabla g_j(\x)\Vert\,\neq\,0,\quad j=1,\ldots,m,
 \end{equation}
 and therefore, with $\x\mapsto t_j(\x):=\Vert\nabla g_j(\x)\Vert^2$,
 \begin{equation}
 \label{tj-bound}
     t_j(\x)\,\geq\,a_j\,,\qquad\forall \x\in\om_j\,,;\quad j=1,\ldots,m,
 \end{equation}
 for some $a_j>0$, $j=1,\ldots,m$.
 
Developing \eqref{stokes-0} yields:
\begin{equation}
\label{stokes-1}
\int_\om {\rm Div}(X)\, p(\x) +\langle X,\nabla p(\x) \rangle\,d\x\,=\,\int_{\partial\om}\langle X,\vec{n}_\x\,\rangle \, p(\x) \,d\sigma(\x) \,.
\end{equation}
Next, select the vector field $X=\x$ and note that 
$\langle \x,\nabla \x^\alpha \rangle = \vert\alpha\vert \x^\alpha$  for all $\alpha\in\N^n$.
%
Then with $p(\x) = \x^\alpha$ in \eqref{stokes-1} with $\alpha \in \N^n$ arbitrary:
\begin{eqnarray}
\label{stokes-1a}
(n+\vert\alpha\vert)\,\int_\om \x^\alpha\,d\x&=&
\sum_{j=1}^m
\int_{\om_j}\left\langle \x,\frac{\nabla g_j(\x)}{\Vert\nabla g_j(\x)\Vert}\right\rangle\, \x^\alpha\,d\sigma_j(\x),\\
&=&
\sum_{j=1}^m
\int_{\om_i}\x^\alpha\,\langle\x,\nabla g_j(\x)\rangle\,\frac{d\sigma_j(\x)}{\Vert \nabla g_j(\x)\Vert},
\end{eqnarray}
where $\sigma_j$ is the restriction to $\om_j$ of the Hausdorff measure $\sigma$ on $\partial\om$.
%

When $\alpha = 0$, then \eqref{stokes-1} has a simple geometric interpretation, easy to visualize in dimension $n=2$. Indeed with $0\in\,{\rm int}(\om)$,
(i) $\langle \x,\vec{n}_\x\rangle$ is the ``height" from $0$ to the hyperplane tangent to $\om$ at the point $M\in\partial\om$ (with coordinate $\x$), and
(ii) $d\sigma(\x)$ is the infinitesimal ``length" $[M,M']$ on $\partial\om$ around $M$, and so $\frac{1}{n}\langle\x,\vec{n}_\x\rangle d\sigma(\x)$
is the infinitesimal ``area" of the  triangle $(O,M,M')$, that is the length of base $[M,M']$ times the height $[0,M]$; see Figure \ref{fig}.
\if{
\begin{figure*}[ht]
{\includegraphics[width=45mm]{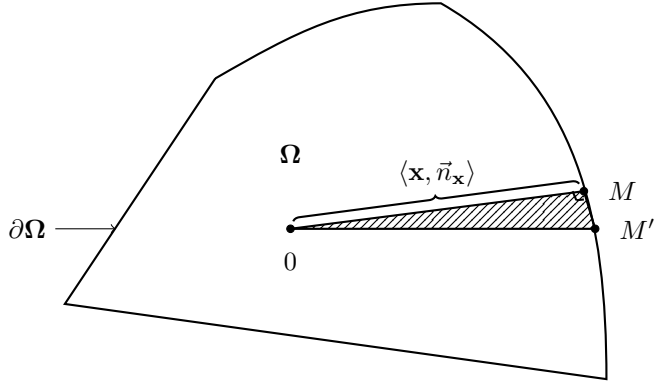}}
\caption{Geometric interpretation \label{fig}}
\end{figure*}
}\fi
\begin{figure}[h!]
  \begin{center}
    \begin{tikzpicture}
      \draw [thick] (-1,2) 
      to [out=30,in=180] (2,3)
      to [out=-30,in=90] (4.2,-2) 
      to (-3,-1)
      to  (-1,2) ;
      \draw [black,fill] (0,0) circle [radius=0.05] node (O) [black,below=0.2] {$0$}; 
            \draw [black,fill] (4.05,0) circle [radius=0.05] node (M') [black,right=0.2]  {$M'$};
            \draw [black,fill] (3.9,0.5) circle [radius=0.05] node (M) [black,right=0.2]  {$M$};
      \draw  (0,1) node  {$\om$};
	 \draw  (-3.5,0) node (A) {$\partial \om$};
	 \draw (-2.2,0) node (B) {};
	 \draw [->] (A) -- (B);
     \draw [thick] (0,0) -- (4.05,0); 
     \draw [thick] (0,0) -- (3.9,0.5);
     \draw [thick] (4.05,0) -- (3.9,0.5);
     \filldraw[color=black,pattern=north east lines] (0,0) -- (4.05,0) -- (3.9,0.5) ;
    \draw [thick] (3.75,0.5) -- (3.8,0.38);
    \draw [thick] (3.8,0.38) -- (3.9,0.39);
     \draw [thick,decoration={brace,raise=0.1cm},
    decorate
] (0.05,0) -- (3.85,0.5) 
node [pos=0.5,anchor=north,yshift=0.8cm] {$\langle\x,\vec{n}_\x\rangle$}; 
    \end{tikzpicture}
  \end{center}
  \caption{Geometric interpretation \label{fig}
}
\end{figure}

There is also a non geometric interpretation. When the $g_j$'s are polynomials then so are the functions $\x\mapsto \langle\x,\nabla g_j(\x)\rangle$'s, which
yields a simple interpretation for \eqref{stokes-1}. Indeed \eqref{stokes-1} states that for each $\alpha\in\N^n$, the moment $\int_\om \x^\alpha d\x$ is some {\em  linear combination} of moments of the 
measure $d\hat{\sigma}=fd\sigma$ on $\partial\om$, with density $f(\x):=\Vert \nabla g_j(\x)\Vert^{-1}$ on $\om_j$, $j=1,\ldots,m$,
(recall \eqref{nondegeneracy}).

\section{The Hausdorff boundary measure of a semi-algebraic set}
We make the following technical assumption on the polynomials $g_j$ that define the boundary $\partial\om$.
\begin{assumption}
\label{ass-2}
Let $\om_j$ be as in \eqref{eq:boundary} and let $\sigma_j$ be the restriction to $\om_j$ of the Hausdorff measure $\sigma$ on $\partial\om$, $j=1,\ldots,m$. Then for every $j=1,\ldots,m$,
$\sigma_j(\{\x: g_k(\x)=b_k\})=0$ for all $k\neq j$.
\end{assumption}

\subsection{The homogeneous case}
\label{sec:homogeneous}
If $g_j$ is homogeneous of degree $d_j$, for each $j=1,\dots,m$, then
$\langle\x,\nabla g_j(\x)\rangle =d_j\,g_j(\x)$ for all $\x \in \R^n$. Therefore, one has:
\begin{equation}
\label{stokes-homog}
(n+\vert\alpha\vert)\,\int_\om \x^\alpha\,d\x\,=\,
\sum_{j=1}^m d_j\,b_j
\int_{\om_j}\x^\alpha\,\frac{d\sigma_j(\x)}{\Vert \nabla g_j(\x)\Vert},\quad\alpha\in\N^n \,.
\end{equation}
Recall that $\lambda_{\om}$ is the restriction of the Lebesgue measure on $\om$.
\subsection*{A first observation} 
\begin{lem}
Let $\om$ in \eqref{set-om} be compact with $b_j>0$ and where $g_j\in\R[\x]$ is homogeneous of degree $d_j$, $j=1,\ldots,m$. Let $\y=(y_\alpha)_{\alpha\in\N^n}$ be the sequence of moments of $\lambda_{\om}$. 

Then the sequence $\hat{\y}=(\hat{y}_\alpha)_{\alpha\in\N^n}=( (n+\vert\alpha\vert) y_\alpha )_{\alpha\in\N^n}$ is 
the sequence of moments of a boundary measure $\phi$ on $\partial\om$, absolutely continuous with respect to 
the Hausdorff measure $\sigma$ on $\partial\om$, which reads:
\begin{eqnarray}
\label{lem-1}
d\phi(\x)&=&\sum_{j=1}^m \frac{b_j\,d_j}{\Vert \nabla g_j(\x)\Vert} \,d\sigma_j(\x),\\
\nonumber
&=&\frac{b_j\,d_j}{\Vert \nabla g_j(\x)\Vert} \,d\sigma(\x),\quad\forall \x\in\om_j,\quad j=1,\ldots,m.
\end{eqnarray}
\end{lem}
\begin{proof}
From \eqref{stokes-homog} we obtain
\begin{eqnarray*}
(n+\vert\alpha\vert) y_\alpha  &=&
\sum_{j=1}^m \int_{\om_j}\x^\alpha\,\frac{d_jb_j\,d\sigma_j(\x)}{\Vert \nabla g_j(\x)\Vert}\\
&=&
\int_{\partial\om}\x^\alpha\,\underbrace{d\left(\sum_{j=1}^m\frac{b_j\,d_j}{\Vert\nabla g_j(\x)\Vert}\,\sigma_j\right)}_{d\phi}(\x)\end{eqnarray*}
The measure $d\phi=\sum_{j=1}^n \frac{b_jd_j}{\Vert \nabla g_j(\x)\Vert} d\sigma_j$ is a positive measure on 
$\partial\om$, absolutely continuous w.r.t. $\sigma$.
\end{proof}
To recover the surface measure $\sigma$ from $\y=(y_\alpha)_{\alpha\in\N^n}$ we proceed in two steps: Define
\begin{equation}
\label{def-phi-j}
d\phi_j(\x)\,=\,\frac{d\sigma_j(\x)}{\Vert \nabla g_j(\x)\Vert},\quad j=1,\ldots,m,\end{equation}
so that by \eqref{stokes-1} and \eqref{lem-1}, for every $\alpha\in\N^n$:
\if{
\begin{equation}
\label{stokes-2}
(n+\vert\alpha\vert)\,y_\alpha\,=\,\sum_{j=1}^m b_jd_j\,\int_{\om_j}
\x^\alpha\,d\phi_j(\x)\,=\,\int_{\partial\om} \x^\alpha\,d\left(\sum_{j=1}^mb_jd_j\,\phi_j\right).
\end{equation}
}\fi
\begin{equation}
\label{stokes-2}
(n+\vert\alpha\vert)\,y_\alpha\,=\,\sum_{j=1}^m b_jd_j\,\int_{\om_j}
\x^\alpha\,d\phi_j(\x)\,=\,\int_{\partial\om} \x^\alpha \sum_{j=1}^md_jb_j\,  d\phi_j .
\end{equation}

\subsection*{Step 1} In Step 1 we compute  
moments of the measures $\phi_j$ in \eqref{def-phi-j}.
Consider the following infinite-dimensional LP:
\if{
\begin{equation}
\label{lp-1}
\begin{array}{rl}
\rho=\displaystyle\inf_{\mu_j}&  \,\displaystyle\sum_{j=1}^m\int_{\om_j} d\mu_j\\
\text{s.t.} & (n+\vert\alpha\vert)\, \displaystyle\int_{\om_j} \x^\alpha\,d\mu  \,=\,
\displaystyle\sum_{j=1}^m d_j\,b_j\,\displaystyle\int_{\om_j} \x^\alpha\,d\mu_j,\quad\forall\alpha\in\N^n \,,\\
&\mu_j\in\mathscr{M}(\om_j)_+,\:j=1,\ldots,m \, .
\end{array}
\end{equation}
}\fi
\begin{align}
\label{lp-1}
\nonumber
\rho=\displaystyle\sup_{\mu,\mu_j}&  \ \int_{\om} d\mu\\ \nonumber
\text{s.t.} & \ \mu \leq \lambda_\B  \,, \\ 
& \ (n+\vert\alpha\vert) \displaystyle\int_{\om} \x^\alpha d\mu =
\displaystyle\sum_{j=1}^m d_j\,b_j\,\displaystyle\int_{\om_j} \x^\alpha   d\mu_j, \, \alpha\in\N^n  \,, \\ \nonumber
& \ \mu\in\mathscr{M}_+(\om)\,, \mu_j\in\mathscr{M}_+(\om_j),\:j=1,\ldots,m \, .
\end{align}
\begin{thm}
\label{first-boundary}
Assume that $b_j>0$ for all $j=1,\ldots,m$ and let Assumption \ref{ass-2} hold.
Then $(\lambda_{\om}, \phi_1,\ldots,\phi_m)$ with $\phi_j$ on $\om_j$ as in \eqref{def-phi-j},
$j=1,\ldots,m$, is the unique optimal solution of LP~\eqref{lp-1} and $\rho = \vol \, (\om)$,
\end{thm}
\begin{proof}
From~\cite[theorem~3.1]{HLS09vol}, the measure $\lambda_{\om}$ is the unique optimal solution of the following infinite-dimensional LP problem:
\begin{equation}
\label{volume} 
\rho' = \sup_{\mu}\: \{\,\int_{\om} d\mu\,:\: \mu \leq \lambda_\B\,;\:\mu \in \mathscr{M}_+(\om) \,\}.
\end{equation}
and $\rho' = \vol \, (\om)$.
In addition $\rho' \geq \rho$. 
Since $(\lambda_{\om}, \phi_1,\ldots,\phi_m)$ 
is a feasible solution of LP~\eqref{lp-1}, one has $\rho \geq \rho'$. Thus $\rho = \rho' = \vol \, (\om)$.

Let $(\mu, \mu_1,\ldots,\mu_m)$ be another optimal solution for LP~\eqref{lp-1}. Then $\mu$ is an optimal solution of LP~\eqref{volume}, yielding $\mu = \lambda_{\om}$.
In addition:
\begin{equation}
\label{eq:vanish}
0=\displaystyle\sum_{j=1}^m d_j\,b_j\,\displaystyle\int_{\om_j} \x^\alpha \,d(\phi_j-\mu_j),\quad \forall \alpha\in\N^n.
\end{equation}
Let $j \in \{1,\ldots,m\}$ and define $\x\mapsto h_j(\x):=\prod_{l\neq j}(b_l-g_l(\x))$. Then by \eqref{eq:vanish},
\begin{equation}
\label{aux}
0 \,=\, \sum_{j=1}^m d_j \, b_j \int_{\om_j}\x^\beta\, h_j(\x)\,d(\phi_j-\mu_j) \,=\, d_j b_j \int_{\om_j}\x^\beta\, h_j(\x)\,d(\phi_j-\mu_j) \,,
\end{equation}
for all $\beta\in\N^n$.
Therefore as \eqref{aux} holds for all $\beta\in\N^n$ and $\om_j$ is compact:
\[ h_j\,d\phi_j\,=\,h_j\,d\mu_j,\quad \forall j=1,\ldots,m.\]
Moreover by Assumption \ref{ass-2}, $h_j(\x)\neq0$, for $\phi_j$-almost all $\x\in\om_j$. 
Hence $\mu_j=\phi_j+\eta_j$ where $\eta_j$ is a measure supported on 
$\{\x\in\om_j: h_j(\x)=0\}$ and therefore $\phi_j(\om_j)\leq\mu_j(\om_j)$. Suppose that 
$\eta_j(\om_j)>0$ for some $j$ so that $\phi_j(\om_j)<\mu_j(\om_j)$. Then \eqref{eq:vanish} with $\alpha=0$ yields the contradiction
\[0\,=\,\sum_{j=1}^m d_j\,b_j\,(\phi_j(\om_j)-\mu_j(\om_j))\,<\,0,\]
and so $\mu_j=\phi_j$ for all $j=1,\ldots,m$.  That is, $(\phi_1,\ldots,\phi_m)$ is the unique feasible (and optimal) solution of \eqref{lp-1}.
\end{proof}
Since LP~\eqref{lp-1} is an infinite-dimensional problem, for practical computation we need  consider finite-dimensional relaxations of this problem. 
Recall that $r_j = \lceil d_j/2\rceil$, for all $j=1,\dots,m$. 
Let us define $d_{\min}^1 := \max \{1, r_1,\dots,r_m \}$.
In practice, fix an integer $d \geq d_{\min}^1$ and consider the following semidefinite program (SDP):
\begin{equation}
\label{sdp-1}
\begin{array}{rl}
\rho_{d}=\displaystyle\sup_{\y, \v_j}& \{\,y_0 \,:\\
\mbox{s.t.}& 
\quad (n+\vert\alpha\vert) \, L_{\y}(\x^\alpha)   = \displaystyle\sum_{j=1}^m d_j\,b_j\, L_{\v_j}(\x^\alpha ) ,\quad \vert\alpha\vert\leq 2d  \,, \\
& \quad \M_d(\y^{\B}) \succeq \M_d(\y)\succeq0 \,, \: \M_{d-r_j}((b_j-g_j)\,\y) \geq 0 \,, \\
& \quad \M_d(\v_j)\succeq0 \,, \: \M_{d-r_j}((b_j-g_j)\,\v_j) = 0 \,,\quad j=1,\dots,m \,,\\
& \quad \M_{d-r_l}((b_l-g_l)\,\v_j)) \succeq 0 \,, \quad l\neq j \,, \quad l,j=1,\dots,m \,\}\,,
\end{array}
\end{equation}
where $\y=(y_\alpha)_{\alpha\in\N^n_{2d}}$, and $\v_j=(v_{j,\alpha})_{\alpha\in\N^n_{2d}}$, $j=1,\ldots,m$.

Of course $(\rho_d)_{d\in\N}$ is a monotone non increasing sequence and the next result shows that by solving the hierarchy of semidefinite programs \eqref{sdp-1}, one recovers the desired solution asymptotically.

\begin{thm}
\label{th-first-step}
Let $(\phi_1,\ldots,\phi_m)$ be as in~\eqref{def-phi-j}.
For each $d \geq d_{\min}^1$, the semidefinite program \eqref{second-step} has an optimal solution $(\y^d,\v_1^d,\dots,\v_m^d)$. 
In addition 
\begin{align}
\label{th-first-step-1}
\lim_{d\to\infty}\,y^d_{\alpha}\,= & \,\int_{\om}\x^\alpha\,d \lambda ,\quad\forall \alpha\in\N^n \,, \\
\lim_{d\to\infty}\, v^d_{j, \alpha}\,= & \,\int_{\om_j}\x^\alpha\,d \phi_j,\quad\forall \alpha\in\N^n;\quad  j=1,\ldots,m.
\end{align}
In particular, as $d\to\infty$, $y_0^d \,\downarrow\,\rho\,=\, \vol (\om)$.
\end{thm}
For ease of exposition the proof is postponed to \S \ref{proof-3.4}.
\begin{rem}
In the case where one already knows moments $\y=(y_\alpha)_{\alpha\in\N^n}$ of the Lebesgue $\lambda_\om$ on $\om$ (e.g. from measurements), then
in \eqref{lp-1}, the left-hand-side in the moment equality constraints
is now the constant $(n+\vert\alpha\vert)\,y_\alpha$. 
One then replaces the criterion $\sup\int_\om d\mu$ with  e.g.
$\sup\sum_j\mu_j(\om_j)$ or $\inf\sum_j\mu_j(\om_j)$. In fact, under Assumption \ref{ass-2}  the feasible set is the singleton $(\phi_1,\ldots,\phi_m)$. 

The same modification is done in the semidefinite relaxations \eqref{sdp-1} and if one chooses  $\sup_{\v_j}\sum_j v_{j,0}$ as criterion then in Theorem \ref{th-first-step}, as $d\to\infty$, \[\displaystyle\sum_{j=1}^m v^d_{j,0}\,\downarrow\,\displaystyle\sum_{j=1}^m\phi_j(\om)\,=\,\lim_{d\to\infty}\rho_d.\]
On the other hand, if one chooses $\inf_{\v_j}\sum_j v_{j,0}$ as criterion then as $d\to\infty$, \[\displaystyle\sum_{j=1}^m v^d_{j,0}\,\uparrow\,\displaystyle\sum_{j=1}^m\phi_j(\om)\,=\,\lim_{d\to\infty}\rho_d.\]
\end{rem}

\subsection*{Step 2} 
In a second step we extract the boundary measure $\sigma_j$ on $\om_j$ 
from the measure $\phi_j$, for every $j=1,\ldots,m$. To do so we use
its moments $\v_j=(v_{j,\alpha})_{\alpha\in\N^n}$, obtained in Step 1.
For each $j=1,\ldots,m$, define the set $\Theta_j\subset\om_j\times \R_+$ by:
\begin{equation}
\label{big-theta}
\Theta_j :=\,
\{ \, (\x,z) \in \om_j \times \R_+: \underbrace{z^2-\Vert \nabla g_j(\x)\Vert^2}_{\theta_j(\x,z)}=0\,\},\quad j=1,\ldots,m \,.
\end{equation}
Observe that if $z^2=\Vert \nabla g_j(\x)\Vert^2$ and $z\geq0$, then $z=\Vert\nabla g_j(\x)\Vert$.
So let $\psi_j$ be a measure on $\Theta_j$ with marginal $\psi_{j,\x}=\phi_j$ on $\om_j$,
and conditional $\hat{\psi}_j(dz\vert \x)$ on $\R_+$. Then disintegrating $\psi_j$ yields:
\begin{eqnarray}
\label{disintegration}
\int_{\Theta_j} \x^\alpha\,z\,d\psi_j(\x,z)&=&\int_{\om_j}\x^\alpha\,\left(\int_{\R_+}z\,\hat{\psi}_j(dz\vert\x)\right)\,d\phi_j(\x)
\nonumber \\
&=&\int_{\om_j}\x^\alpha\,\Vert \nabla g_j(\x)\Vert\,d\phi_j(\x)\\
\nonumber
&=&\int_{\om_j}\x^\alpha\,d\sigma_j(\x),\quad\forall\alpha\in\N^n\quad\mbox{[by \eqref{def-phi-j}]} \,.
\end{eqnarray}

Recall that $\ell_j={\rm deg}(t_j)/2={\rm deg}(\theta_j)/2= d_j-1$ and let
$(v_{j,\alpha})_{\alpha\in\N^n}$ be all moments of $\phi_j$ obtained in step 1. As $\om_j$ is compact, let $\tau >\sup_{\x\in\om_j}\Vert \nabla g_j(\x)\Vert^2$,
so that one may and will impose the additional redundant constraint  $z^2\leq \tau$.
Then for each $j=1,\ldots,m$, consider the 
hierarchy of semidefinite programs indexed by $d\in\N$:
\begin{equation}
\label{second-step}
\begin{array}{rl}
\overline{\rho}_{j,d}=\displaystyle\sup_{\u}& \{\, u_{0,1}:\\
\mbox{s.t.}& u_{\alpha,0}=v_{j,\alpha},\quad \vert\alpha\vert\leq 2d \,,\\
&\M_{d}(\u)\,\succeq\,0 \,, \:\M_{d-\ell_j}(\theta_j\,\u)\,=\,0 \,,\\
&\M_{d-1}((\tau-z^2)\,\u)\succeq0 \,,\:\M_{d-1}(z\,\u)\succeq0 \,\},
\end{array}
\end{equation}
and
\begin{equation}
\label{second-step-inf}
\begin{array}{rl}
\underline{\rho}_{j,d}=\displaystyle\inf_{\u}& \{\, u_{0,1}:\\
\mbox{s.t.}& u_{\alpha,0}=v_{j,\alpha},\quad \vert\alpha\vert\leq 2d \,,\\
&\M_{d}(\u)\,\succeq\,0 \,, \:\M_{d-\ell_j}(\theta_j\,\u)\,=\,0 \,,\\
&\M_{d-1}((\tau-z^2)\,\u)\succeq0 \,,\:\M_{d-1}(z\,\u)\succeq0 \,\},
\end{array}
\end{equation}
where $\u=(u_{\alpha,k})_{(\alpha,k)\in\N^{n+1}_{2d}}$. 
Let $d_{\min}^2 := \max \{1,r_1,\dots,r_m,\ell_j\}$.

\begin{thm}
\label{th-second-step}
If SDP~\eqref{second-step} (resp. SDP \eqref{second-step-inf}) has a feasible solution for each $d \geq d_{\min}^2$, then it has an optimal solution $\u^d=(u^d_{\alpha,k})$ (resp. $\w^d=(w^d_{\alpha,k})$). 
In addition:
\begin{equation}
\label{th-second-step-1}
\lim_{d\to\infty}\,u^d_{\alpha,1}\,=\,\lim_{d\to\infty}\,w^d_{\alpha,1}\,=\,\int_{\om_j}\x^\alpha\,d\sigma_j,\quad\forall \alpha\in\N^n.
\end{equation}
In particular, $\overline{\rho}_{j,d}=u^d_{0,1}\downarrow \sigma_j(\om)$ and $\underline{\rho}_{j,d}=u^d_{0,1}\uparrow \sigma_j(\om)$, 
as $d\to\infty$.
\end{thm}
For ease of exposition the proof is postponed to \S \ref{proof-3.5}. 

\begin{rem}
\label{feasible}
If $\v_j=(v_{j,\alpha})_{\alpha\in\N^n_{2d}}$ is the vector of moments of $\phi_j$ on $\om_j$ (up to degree $2d$),
then \eqref{second-step} and \eqref{second-step-inf} have a feasible solution. 
It suffices to consider the moments $\u=(u_{\alpha,k})_{(\alpha,k)\in\N^{n+1}_{2d}}$ of the measure 
$d\phi(\x,z)=\delta_{\Vert \nabla g_j(\x)\Vert} d\phi_j(\x)$, where $\delta_\bullet$ is the Dirac measure. Indeed such a vector $\u$ is feasible by construction. In fact in this case, an infinite sequence $\u=(u_{\alpha,k})_{(\alpha,k)\in\N^{n+1}}$ that satisfies \emph{all} constraints of \eqref{second-step} is unique and is the moment sequence of the measure $\delta_{\Vert \nabla_j(\x)\Vert}\,\phi_j$, and so \eqref{th-second-step-1} holds.
\end{rem}

Theorem \ref{th-second-step} states that by solving the hierarchy of semidefinite programs \eqref{second-step}, one may approximate as closely as desired 
any finite number of moments of the surface measure $\sigma_j$. In addition, depending on whether one maximizes or minimizes the same criterion $u_{0,1}$, one obtains a monotone sequence of upper bounds or lower bounds that converges to $\sigma_j(\om_j)$.  
After summing up, this allows to obtain smaller and smaller approximations
\[\sum_{j=1}^m \underline{\rho}_{j,d}\,\leq\,\sigma(\partial\om)\,\leq\,\sum_{j=1}^m\overline{\rho}_{j,d}\]
of $\partial\om$ (length if $n=2$ or area if $n=3$).\\

We next treat the general case where the $g_j$'s are not necessarily homogeneous. Recall \eqref{stokes-1a} which reads for all $\alpha \in \N^n$:
\begin{align}
\label{stokes-convex}
& (n+\vert\alpha\vert)\,\int_\om \x^\alpha\, d\x\,     
= \sum_{j=1}^m\int_{\om_j}\x^\alpha\,\underbrace{\langle \x,\nabla g_j(\x)\rangle}_{=:q_j(\x)} \,\underbrace{\frac{d\sigma_j(\x)}{\Vert \nabla g_j(\x)\Vert}}_{d\phi_j(\x)} \,.
\end{align}
In the remaining part of the paper we make the following assumption:
\begin{assumption}
\label{ass-exp}
Let $\phi_j\in\mathscr{M}_+(\om_j)$ and $q_j\in\R[\x]$ be as in \eqref{stokes-convex} and recall that 
$\Vert \nabla g_j\Vert\neq0$ on $\om_j$ for all $j=1,\ldots,m$.
For every $j=1,\ldots,m$, let $\x\mapsto h_j(\x):=\prod_{k\neq j}(b_k-g_k(\x))$.
Then for all $j=1,\ldots,m$, $\sigma(\{\x\in\om_j: q_j(\x)\,h_j(\x)=0\})=0$.
\end{assumption}
\subsection{The general convex case}
\label{sec:convex}
Let $\om$ be as in \eqref{set-om} and assume that $\om$ is convex and $g_j(0)<b_j$ for all $j=1,\ldots,m$.
This time we can exploit the fact that $\langle \vec{n}_\x,\x\rangle\geq0$ for all $\x\in\partial\om$, and therefore:
\begin{equation}
\label{convexity}
\x\mapsto q_j(\x)\,:=\,\langle \x,\nabla g_j(\x)\rangle\,\geq0,\qquad\forall \x\in\om_j \,, \quad j=1,\ldots,m \,.
\end{equation}

\if{
In particular for $k = 0$, one has for all $\alpha \in \N$:
\begin{equation}
\label{stokes-convex-0}
(n+\vert\alpha\vert)\,\int_\om \x^\alpha\,d\x\, =\,\sum_{j=1}^m\int_{\om_j}\x^\alpha\, q_j(\x) \,d\phi_j(\x) \,.
\end{equation}
}\fi
Incidentally \eqref{stokes-convex} has the following nice interpretation:
\begin{cor}
Let $\om$ be as in \eqref{set-om} with $g_j(0)<b_j$, $j=1,\ldots,m$, and assume that all $g_j$'s are continuously differentiable. If the $g_j$'s are all positively homogeneous functions or if $\om$ is convex then
the real sequence $\hat{\y}=(\hat{y}_\alpha)_{\alpha\in\N^n}$ where 
$\hat{y}_\alpha=(n+\vert\alpha\vert)\,y_\alpha$, $\alpha\in\N^n$, has a representing measure 
$d\phi=f\,d\sigma$ on $\partial\om$, hence
absolutely continuous w.r.t. $\sigma$, and with density:
\begin{equation}
f(\x)\,=\,\frac{\langle \x,\nabla g_j(\x)\rangle}{\Vert \nabla g_j(\x)\Vert},\quad \x\in\om_j, \quad j=1,\ldots,m \,.
\end{equation}
\end{cor}
Hence the infinite-dimensional LP of Step 1 in Section~\ref{sec:homogeneous} is now replaced with:
\begin{align}
\label{lp-convex-1}
\nonumber
\rho=\displaystyle\sup_{\mu,\mu_j} &  \:\{\, \int_{\om} d\mu\,:\\ \nonumber
\text{s.t.} & \ \mu \leq \lambda_\B  \,, \\ 
& \ (n+\vert\alpha\vert) \displaystyle\int_{\om} \x^\alpha  \, d\mu  = \displaystyle\sum_{j=1}^m \displaystyle\int_{\om_j} \x^\alpha  \, q_j(\x) \, d\mu_j, \, \alpha\in\N^n \,, \\ \nonumber
& \ \mu\in\mathscr{M}_+(\om) \,, \mu_j\in\mathscr{M}_+(\om_j),\:j=1,\ldots,m \, .
\end{align}
\if{
\begin{equation}
\label{lp-convex-1}
\begin{array}{rl}
\rho=\displaystyle\sup_{\mu,\mu_j}&\{ \,\displaystyle\int d\mu: \:
(n+\vert\alpha\vert)\,y_\alpha\,=\,
\displaystyle\sum_{j=1}^m \displaystyle\int \x^\alpha\,q_j(\x)\,d\mu_j,\quad\forall\alpha\in\N^n\\
&\mu_j(\om_k)=0,\quad k\neq j;\:j=1,\ldots,m\\
&\mu_j\in\mathscr{M}(\om_j)_+,\:j=1,\ldots,m\,\}.
\end{array}
\end{equation}
}\fi
Under Assumption \ref{ass-exp}, as  before one can show that this infinite-dimensional LP~\eqref{lp-convex-1} has a unique solution $(\lambda_{\om}, \phi_1,\ldots,\phi_m)$ with $d\phi_j=\Vert \nabla g_j(\x)\Vert^{-1}\,d\sigma_j$ for all $j=1,\ldots,m$.
The proof, very similar to that of Theorem \ref{th-first-step} is omitted.

Finally, to recover $\sigma_j$ from $\phi_j$, we again solve the hierarchy of semidefinite programs  \eqref{second-step} of Step 2.

\subsection{Discussion}

One may also approximate moments of the $\sigma_j$'s by solving a \emph{single} hierarchy that combines
the constraints of \eqref{sdp-1} and \eqref{second-step}, that is, by solving the following hierarchy of semidefinite programs indexed by $d\in\N$:
\begin{equation}
\label{sdp-single-step}
\begin{array}{rl}
\rho_{d}=\displaystyle\sup_{\y, \v_j,\u^j}& \{\,y_0 \,:\\
\mbox{s.t.}& 
\quad (n+\vert\alpha\vert) \, L_{\y}(\x^\alpha)   = \displaystyle\sum_{j=1}^m d_j\,b_j\, L_{\v_j}(\x^\alpha ) ,\quad \vert\alpha\vert\leq 2d  \,, \\
& \quad \M_d(\y^{\B}) \succeq \M_d(\y)\succeq0 \,, \: \M_{d-r_j}((b_j-g_j)\,\y) \geq 0 \,, \\
& \quad \M_d(\v_j)\succeq0 \,, \: \M_{d-r_j}((b_j-g_j)\,\v_j) = 0 \,,\quad j\leq m \,,\\
& \quad \M_{d-r_l}((b_l-g_l)\,\v_j)) \succeq 0 \,, \quad l\neq j \,, \quad j\leq m \,\\
&\\
& \quad u^j_{\alpha,0}=v_{j,\alpha},\qquad \vert\alpha\vert\leq 2d \,,\:j\leq m\,,\\
&\quad \M_{d}(\u^j)\,\succeq\,0 \,, \:\M_{d-\ell_j}(\theta_j\,\u^j)\,=\,0 \,,\quad j\leq m\,,\\
&\quad \M_{d-1}((\tau-z^2)\,\u^j)\succeq0 \,,\:\M_{d-1}(z\,\u^j)\succeq0 \,,\quad j\leq m\,\}\,.
\end{array}
\end{equation}
Indeed recall Remark \ref{feasible}. If $\y=(y_\alpha)_{\alpha\in\N}$ is the moment sequence
of $\lambda_\om$ then an infinite sequence $\u^j=(u^j_{\alpha,k})_{(\alpha,k)\in\N^{n+1}}$ that satisfies all constraints of \eqref{second-step}, for all $d\in\N$, is unique and is the moment sequence of the measure $\delta_{\Vert\nabla g_j(\x)\Vert}\phi_j$ on $\om_j$. 
So the semidefinite program \eqref{sdp-single-step} has always a feasible solution. 
However it has $m$ additional 
unknown moment sequences $(\u^j)_{j\leq m}$, hence is harder than \eqref{sdp-1} to solve, and the numerical results can be less accurate (see Section~\ref{sec:benchs}).
If $m$ is small it still may be an interesting alternative to the two-step procedure.
On the other hand, one looses the upper and lower bounds $\overline{\rho}_{j,d}$ and $\underline{\rho}_{j,d}$ obtained in \eqref{second-step} and \eqref{second-step-inf} respectively.  
However, for the general case treated in the next section we \emph{cannot} avoid a two-step procedure.

\subsection{The general case}
\label{sec:general}

The general case is more delicate because in the integrand of the right-hand-side
of Stokes' identity \eqref{stokes-convex}, the term 
$q_j(\x):=\langle \x,\nabla g_j(\x)\rangle$ may change sign on $\om_j$. 
So to handle the case where $\langle\vec{n}_\x,\x\rangle <0$ on some Borel set
$B\subset\partial\om$ with $\sigma(B)>0$,
then write $\phi_j=\phi_j^++\phi_j^-$, $j=1,\ldots,m$,  with:
\begin{eqnarray*}
{\rm supp}(\phi_j^+)&=&\{\x\in\om_j:\:q_j(\x)\,\geq0\}\,=:\,\om_j^+\\
{\rm supp}(\phi_j^-)&=&\{\x\in\om_j:\:q_j(\x)\,<0\}\,=:\,\om_j^-.\end{eqnarray*}

We now assume that the moments $\y=(y_\alpha)_{\alpha\in\N^n}$ of $\lambda_\om$ are already available (e.g. from some measurements) or they have
been already computed, e.g. by solving the infinite-dimensional LP \eqref{volume}. In practice in both cases 
only approximations of finitely many of them are available either by measurements or by an approximation algorithm (e.g., as the one described in \cite{HLS09vol}).

\subsection*{Step 1}
The infinite-dimensional LP of Step 1 in \eqref{lp-convex-1} now becomes:
\begin{align}
\label{lp-general-1}
\rho=\displaystyle\inf_{\mu^\pm_j}&  \,\Bigg\{\, \sum_{i=1}^n\sum_{j=1}^m 
\left(
\int_{\om_j^+}d\mu_j^++\int_{\om_j^-} d\mu^-_j \right) \,:\\ \nonumber
\text{s.t.} & \ (n+\vert\alpha\vert) \,y_\alpha   =
\displaystyle\sum_{j=1}^m \displaystyle\int_{\om_j^\pm} \x^\alpha  \, q_j(\x)  \,d\mu^\pm_j, \quad \alpha\in\N^n  \,, \\ \nonumber
& \ \mu^\pm_j\in\mathscr{M}_+(\om_j^\pm),\:j=1,\ldots,m \,\Bigg\} .
\end{align}
\begin{thm}
\label{th:signed}
Let $\om\subset\B=(-1,1)^n$ and let Assumption \ref{ass-exp} hold. Then
$(\phi_1^\pm,\ldots,\phi_m^\pm)$ is the unique optimal solution of \eqref{lp-general-1}.
\end{thm}

For ease of exposition the proof is postponed to \S \ref{proof-signed}. To obtain moments of the measures $(\phi^\pm_j)_{j\leq m}$,  consider the following hierarchy of semidefinite programs, the analogue for the general case of \eqref{sdp-1} for the homogeneous case. 
\begin{equation}
\label{sdp-2}
\begin{array}{rl}
\rho_{d}=\displaystyle\inf_{(\v_j)}& \{\,\displaystyle\sum_{j=1}^m v_{j,0}^\pm\,:\\
\mbox{s.t.}& 
(n+\vert\alpha\vert) \, y_\alpha   = \displaystyle\sum_{j=1}^m L_{\v^\pm_j}(q_j\,\x^\alpha ) ,\quad \vert\alpha\vert\leq 2d  \,, \\
& \M_d(\v^\pm_j)\succeq0 \,, \: \M_{d-r_j}((b_j-g_j)\,\v^\pm_j) = 0 \,,\quad j=1,\dots,m \,,\\
& \M_{d-r_l}((b_l-g_l)\,\v^\pm_j)) \succeq 0 \,, \quad l\neq j \,, \quad l,j=1,\dots,m \,,\\
&\M_{d-1}((1-x_i^2)\,\v_j^\pm)\,\succeq0,\quad i=1,\ldots,n\,;\:j=1,\ldots,m\,,\\
&\M_{d-\ell_j}(q_j\,\v_j^+)\succeq0\,,\:
\M_{d-\ell_j}(-q_j\,\v_j^-)\succeq0\,\},
\end{array}
\end{equation}
where $\v^\pm_j=(v^\pm_{j,\alpha})_{\alpha\in\N^n_{2d}}$, and
$\ell_j=\lceil ({\rm deg} \,q_j)/2\rceil$, $j=1,\ldots,m$.

Of course $(\rho_d)_{d\in\N}$ is a monotone non increasing sequence and the next result shows that by solving the hierarchy of semidefinite programs \eqref{sdp-1}, one recovers the desired solution asymptotically.

\begin{thm}
\label{th:general-sdp-1}
For every $d$ the semidefinite program \eqref{sdp-2} has an optimal solution $(\v_j^{\pm,d})_{j\leq m}$. In addition:
\begin{equation}
\label{th:general-sdp-1-1}
\lim_{d\to\infty} v_{j,\alpha}^{\pm,d}\,=\,\int_{\om_j^\pm} \x^\alpha\,d\phi^\pm_j,\quad j=1,\ldots,m.
\end{equation}
\end{thm}
For ease of exposition a proof is postponed to \S \ref{proof-general-sdp-1}. Next we consider the problem of recovering
$\sigma$ from moments of the $\phi_j^\pm$.

\subsection*{Step 2} For each $j=1,\ldots,m$, let
$\Theta_j\subset\om_j\times \R_+$ be the set defined in \eqref{big-theta} and recall that
if $z^2=\Vert \nabla g_j(\x)\Vert^2$ and $z\geq0$, then $z=\Vert\nabla g_j(\x)\Vert$.
So let $\psi^\pm_j$ be a measure on $\Theta_j$ with marginal $\psi^\pm_{j,\x}=\phi^\pm_j$ on $\om^\pm_j$,
and conditional $\hat{\psi}^\pm_j(dz\vert \x)$ on $\R_+$. Then disintegrating $\psi^\pm_j$ yields:
\begin{eqnarray}
\label{disintegration_sign}
\int_{\Theta_j} \x^\alpha\,z\,d\psi^\pm_j(\x,z)&=&\int_{\om^\pm_j}\x^\alpha\,\left(\int_{\R_+}z\,\hat{\psi}^\pm_j(dz\vert\x)\right)\,d\phi^\pm_j(\x)
\nonumber \\
&=&\int_{\om^\pm_j}\x^\alpha\,\Vert \nabla g_j(\x)\Vert\,d\phi^\pm_j(\x)\\
\nonumber
&=&\int_{\om^\pm_j}\x^\alpha\,d\sigma_j(\x),\quad\forall\alpha\in\N^n\quad\mbox{[by \eqref{def-phi-j}]} \,.
\end{eqnarray}
Therefore:
\[\int_{\Theta_j} \x^\alpha\,z\,d\psi^+_j(\x,z)
+\int_{\Theta_j} \x^\alpha\,z\,d\psi^-_j(\x,z)\,=\,
\int_{\om_j}\x^\alpha\,d\sigma_j(\x),\quad\forall\alpha\in\N^n.\]

So in Step 2, for every $j=1,\ldots,m$, the goal is to obtain the two moment sequences  $\u^\pm=(u^\pm_{\alpha,k})_{(\alpha,k)\in\N^{n+1}}$. Recall that $\tau>\sup_{\x\in\om_j}\Vert\nabla g_i(\x)\Vert$. To obtain $\u^+$, \eqref{second-step} now reads:
\begin{equation}
\label{second-step-general-+}
\begin{array}{rl}
\underline{\rho}_{j,d}^+=\displaystyle\inf_{\u}& \{\,u_{0,1}:\\
\mbox{s.t.}& u_{\alpha,0}=\phi^+_{j,\alpha},\quad \vert\alpha\vert\leq 2d \,,\\
&\M_d(\u)\,\succeq0,\,,\M_{d-\ell_j}(\theta_j\,\u)\,=\,0  \,, \\
& \M_{d-1}(z\,\u)\succeq0 \,,\,
\M_{d-1}((\tau-z^2)\,\u)\succeq0\,\}.
\end{array}\end{equation}

Similarly, to obtain $\u^-$, one solves:
\begin{equation}
\label{second-step-general--}
\begin{array}{rl}
\underline{\rho}_{j,d}^-=\displaystyle\inf_{\u}& \{\,u_{0,1}:\\
\mbox{s.t.}& u_{\alpha,0}=\phi^-_{j,\alpha},\quad \vert\alpha\vert\leq 2d \,,\\
&\M_d(\u)\,\succeq0,\,,\M_{d-\ell_j}(\theta_j\,\u)\,=\,0  \,, \\
& \M_{d-1}(z\,\u)\succeq0 \,,\,
\M_{d-1}((\tau-z^2)\,\u)\succeq0\,\}.
\end{array}\end{equation}
One may also solve:
\begin{equation}
\label{second-step-general-+-sup}
\begin{array}{rl}
\overline{\rho}_{j,d}^+=\displaystyle\sup_{\u}& \{\,u_{0,1}:\\
\mbox{s.t.}& u_{\alpha,0}=\phi^+_{j,\alpha},\quad \vert\alpha\vert\leq 2d \,,\\
&\M_d(\u)\,\succeq0,\,,\M_{d-\ell_j}(\theta_j\,\u)\,=\,0  \,, \\
& \M_{d-1}(z\,\u)\succeq0 \,,\,
\M_{d-1}((\tau-z^2)\,\u)\succeq0\,\},
\end{array}\end{equation}
and
\begin{equation}
\label{second-step-general---sup}
\begin{array}{rl}
\overline{\rho}_{j,d}^-=\displaystyle\sup_{\u}& \{\,u_{0,1}:\\
\mbox{s.t.}& u_{\alpha,0}=\phi^-_{j,\alpha},\quad \vert\alpha\vert\leq 2d \,,\\
&\M_d(\u)\,\succeq0,\,,\M_{d-\ell_j}(\theta_j\,\u)\,=\,0  \,, \\
& \M_{d-1}(z\,\u)\succeq0 \,,\,
\M_{d-1}((\tau-z^2)\,\u)\succeq0\,\}.
\end{array}\end{equation}

But all those programs are exactly like \eqref{second-step}
and so Theorem \ref{th-second-step} applies. For instance,
let $\u^{+,d}$ (resp. $\u^{-,d}$) be an optimal solution of 
\eqref{second-step-general-+} (resp. \eqref{second-step-general--}). Then by Theorem
\ref{th-second-step},
\[\lim_{d\to\infty}u_{\alpha,1}^{+,d}+u_{\alpha,1}^{-,d}
\,=\,\int_{\om_j^+}\x^\alpha\,
d\sigma_j+\int_{\om_j^-}\x^\alpha\,
d\sigma_j\,=\,\int_{\om_j}\x^\alpha\,d\sigma_j,\]
for all $\alpha\in\N^n$, and moreover $(u_{0,1}^{+,d}+u^{-,d}_{0,1})\uparrow \sigma_j(\om_j)$ as $d\to\infty$.

Similarly, let $\u^{+,d}$ (resp. $\u^{-,d}$) be an optimal solution of 
\eqref{second-step-general-+-sup} (resp. \eqref{second-step-general---sup}). Then by Theorem
\ref{th-second-step} again,
\[\lim_{d\to\infty}u_{\alpha,1}^{+,d}+u_{\alpha,1}^{-,d}
\,=\,\int_{\om_j^+}\x^\alpha\,
d\sigma_j+\int_{\om_j^-}\x^\alpha\,
d\sigma_j\,=\,\int_{\om_j}\x^\alpha\,d\sigma_j,\]
for all $\alpha\in\N^n$, and moreover $(u_{0,1}^{+,d}+u^{-,d}_{0,1})\downarrow \sigma_j(\om_j)$ as $d\to\infty$.

\section{Numerical Experiments}
\label{sec:benchs}
Here, we illustrate our theoretical framework on simple basic compact semi-algebraic sets. 
Our numerical experiments are performed with the Gloptipoly 3~\cite{gloptipoly3} library available in Matlab
We used Mosek 8~\cite{moseksoft} to solve the SDP problems. 
All results were obtained on an Intel Xeon(R) E-2176M CPU (2.70GHz $\times$ 12) with 32Gb of RAM.

%
We first consider the two-dimensional unit disk $\om = \{\x \in \R^2 : \| \x \|_2^2 \leq 1 \}$, corresponding to $g_1 = \| \x \|_2^2$ and $b_1 = 1$. 
The moments of the Hausdorff boundary measure of $\partial \om$ are approximated after solving SDP~\eqref{sdp-1} (first step) and SDP~\eqref{second-step}-\eqref{second-step-inf} (second step), successively.
The second step allows one to compute the absolute error gap between the optimal value $\overline{\rho}_{1,d}$ of SDP~\eqref{second-step} and $\underline{\rho}_{1,d}$ of SDP~\eqref{second-step-inf}, which respectively provide an upper bound and a lower bound on the perimeter of the boundary of $\om$.
To ensure that the moments of the uniform measure on $\om$ are approximated with good accuracy, we solve the first step with relatively large value of $d$, namely with $d = 10$.
These moment approximations are then used as input in the SDP relaxations related to the second step.
At $d = 2$, we already obtain $\overline{\rho}_{1,d} = 6.28319 \gtrsim 2 \pi \gtrsim \underline{\rho}_{1,d} = 6.28317$.

\if{
\begin{table}[!ht]
\caption{Relative and absolute errors obtained when approximating the mass of the Hausdorff measure on the boundary of the unit disk}
\label{table:disk}
\begin{tabular}{c|ccccc}
$d$ & 1 & 2 & 3 & 4 & 5 \\
\hline
relative error & 54\% & 1.86 \% & 0.09 \% & 0.08 \% & 0.01 \%\\
absolute error & & & & & 
\end{tabular}
\end{table}
}\fi
Then, we repeat the same experiments on the so-called "TV screen", defined by $\om = \{\x \in \R^2 : x_1^4 + x_2^4 \leq 1 \}$. 
The approximate value of the perimeter of the boundary is given by numerical integration of $2 \times \int_{-1}^1 \sqrt[4]{1-t^4} d t \simeq 7.0177$. 
In Table~\ref{table:tvscreen}, we display the relative errors in percentage when approximating the perimeter of the boundary of  $\om$, namely the mass of the boundary measure, for increasing values of the relaxation order $d$.
We also display the absolute error gap between the optimal value $\overline{\rho}_{1,d}$ of SDP~\eqref{second-step} and $\underline{\rho}_{1,d}$ of SDP-\eqref{second-step-inf}.
Table~\ref{table:tvscreen} indicates that the quality of the approximations increases significantly when the relaxation order grows.
We also implemented SDP~\eqref{sdp-single-step}, i.e., the relaxation corresponding to a single hierarchy. 
In this case, the approximation of the perimeter is less accurate as we obtain a relative error of $17.9\%$  for $d = 3$, $5.98\%$ for $d = 4$ and $5.65 \%$ for $d = 5$. 
With higher relaxation orders, we encountered numerical issues, certainly due to the growing number of SDP constraints and SDP variables.

\begin{table}[!ht]
\caption{Relative and absolute errors obtained when approximating the mass of the Hausdorff measure on the boundary of the TV screen}
\label{table:tvscreen}
\begin{tabular}{c|ccc}
$d$ & 3 & 4 & 5  \\
\hline
relative error & 0.18 \% & 0.02 \% & 0.01 \%\\
absolute error & 4.37 & 1.5$\text{e}$-3 & 1.4$\text{e}$-5
\end{tabular}
\end{table}

Eventually, we consider the non-convex two-dimensional ``star-shaped'' curve, defined to be the boundary of $\om = \{\x \in \R^2 : x_1^4 + x_2^4 - 1.7 x_1^2 x_2^2 \leq 1 \}$, displayed in Figure~\ref{fig:star}.
Again, using numerical integration scheme, we obtain an approximate value of the perimeter equal to $11.3668$.
By contrast with the two previous examples, Table~\ref{table:star} shows that this is slightly more difficult to obtain accurate approximations for the mass of the boundary measure of this non-convex set, as we need to compute the fifth order relaxation to get a relative error below 0.1 \%.
As for the previous example, the bounds obtained via SDP~\eqref{sdp-single-step} are less accurate since we obtain a relative error of $24.1\%$  for $d = 3$ and $0.36 \%$ for $d = 4$. 

\begin{figure}[ht]
\includegraphics[scale=1]{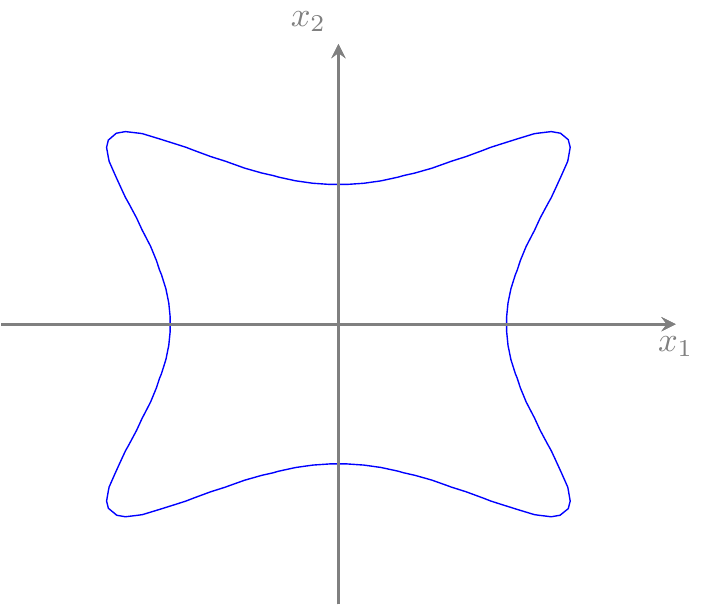}
 \caption{The non-convex two-dimensional ``star-shaped'' curve, given by $\partial \om = \{\x \in \R^2 : x_1^4 + x_2^4 - 1.7 x_1^2 x_2^2 = 1 \}$} \label{fig:star}
\end{figure}

\if{
\begin{figure}[ht]
\begin{tikzpicture}
\begin{axis}[
            xmin=-2, xmax=2, 
            ymin=-2, ymax=2,
            axis x line = center, 
            axis y line = center,
            axis line style = {thick, gray},
            xtick = \empty,
            ytick = \empty,
            xlabel = {$x_1$},
            every axis x label/.append style = {below, gray},
            ylabel = {$x_2$},
            every axis y label/.append style = {above left, gray},
            legend style = {nodes=right},
            legend pos = north east,
            clip mode = individual,
            view = {0}{90},    
        ]
	\addplot3[contour ={
                levels=0, labels=false,
            }, 
            samples=20,
            ] 
            {x^4 + y^4 - 1.7* x^2 * y^2 - 1};
\end{axis}
\end{tikzpicture}
 \caption{The non-convex two-dimensional ``star-shaped'' curve, given by $\partial \om = \{\x \in \R^2 : x_1^4 + x_2^4 - 1.7 x_1^2 x_2^2 = 1 \}$} \label{fig:star}
\end{figure}
}\fi
\begin{table}[!ht]
\caption{Relative and absolute errors obtained when approximating the mass of the Hausdorff measure on the boundary of a ``star-shaped'' curve}
\label{table:star}
\begin{tabular}{c|ccc}
$d$ & 3 & 4 & 5  \\
\hline
relative error &    1.06\% &  0.13\% &  0.09\%\\
absolute error &  0.87 & 0.26 & 0.04
\end{tabular}
\end{table}

\section{Conclusion}
\label{sec:conclusion}
We have provided a general numerical scheme to approximate moments of the Hausdorff boundary measure of a basic compact semi-algebraic set, as closely as desired. To the best of our knowledge, it is the first such contribution at least  at this level of generality. However it should be noted that the size of the matrices involved in the SDP relaxations grows up rapidly with the relaxation order. Therefore so far this approach is  limited to basic semi-algebraic sets of modest dimension. 
On the other hand, if the polynomials involved in the description of the semi-algebraic set follow a specific sparsity pattern, then the approach developed in ~\cite{tacchi2019exploiting} for volume computation in the presence of sparsity, might offer a way to overcome this scalability issue in order to handle semi-algebraic sets of higher dimensions.

\section{Appendix}

\subsection{Proof of Theorem \ref{th-first-step}}
\label{proof-3.4}
\begin{proof}
The feasible set of SDP~\eqref{sdp-1} is closed. In addition the constraint 
$\M_d(\y)\preceq\M_d(\y^\B)$ implies that $L_\y(x_i^{2d})\leq L_{\y^\B}(x_i^{2d})\leq 2^n$ for all $i=1,\ldots,n$,
because $\B=(-1,1)^n$. In addition, $L_\y(1)\leq L_{\y^\B}(1)=2^n$, and so by \cite[Proposition 3.6]{lasserre2009moments}, $\vert y_\alpha\vert\leq 2^n$ for all
$\alpha\in\N^n_{2d}$. Moreover:
\[(n+2d)\,L_\y(x_i^{2d})\,=\,\sum_{j=1}^m d_j\,b_j\,L_{\v_j}(x_i^{2d}),\quad i=1,\ldots,n,\]
and $L_\y(1)=\sum_{j=1}^m d_j\,b_j\,L_{\v_j}(1)$, $j=1,\ldots,n$.

Therefore $\max_i L_{\v_j}(x_i^{2d})\,\leq \frac{n+2d}{d_j\,b_j}\,\max_i L_\y(x_i^{2d})$,
and $L_{\v_j}(1)\leq n\,2^n/(b_jd_j)$ for all $j=1,\ldots,n$. But then by \cite[Proposition 3.6]{lasserre2009moments},
$\sup\,\{\,\vert v_{j,\alpha}\vert\,:\: j\leq m;\, \alpha\in\N^n_{2d}\} <\tau_d$,
where $\tau_d:=2^n\,\max_j\frac{n+2d}{b_jd_j}$. Hence the feasible set is closed and bounded, thus compact, which implies that SDP~\eqref{sdp-1} has an optimal solution $(\y^d,\v_1^d,\dots,\v_m^d)$.

Complete the finite vectors $\y^d,\v_j^d$ with zeros to make them bounded infinite sequences
and define  the sequence $\hat{\y}^d=(\hat{y}^d_\alpha)_{\alpha\in\N^n}$ with $\hat{y}^d_\alpha=y_\alpha/\tau_\ell$ for all 
$2\ell-1\leq\vert\alpha\vert\leq 2\ell$, so that $\sup_\alpha \vert\hat{y}^d_\alpha\vert\leq 1$. Hence every sequence $\hat{\y}^d$ is an element of the unit ball of the space $\ell_\infty$ of uniformly bounded sequences
(also the topological dual of $\ell_1$). 
We can do a similar scaling for the $\v_j^d$, $j=1,\ldots,m$.
As a consequence of Banach-Alaoglu theorem~\cite{Ash72RAP} applied to the unit ball of $\ell_\infty$, there exists a subsequence $\{d_s\}_{s \in \N} \subset \N$ and elements $\hat{\y},\hat{\v}_j$ in $\ell_\infty$ such that 
$\hat{\y}^{d_s} \to \hat{\y}$ and $\hat{\v}_j^{d_s} \to \hat{\v}_j$ ($j=1,\dots,m$) as $s \to \infty$, for the weak-$\star$ topology $\sigma(\ell^\infty,\ell_1)$. In particular,
defining $\y=(y_\alpha)$ with $y_\alpha:=\tau_\ell\,\hat{y}_\alpha$, and
$\v_j=(v_{j,\alpha})$ with $v_{j,\alpha}:=\tau_\ell\,\hat{v}_{j,\alpha}$, for all $2\ell-1\leq\vert\alpha\vert\leq 2\ell$, and all $\ell\in\N$:
\begin{eqnarray}
\label{conv-1}
\lim_{s\to\infty}\,y^{d_s}_{\alpha}&=&y_\alpha,\quad\forall \alpha\in\N^n \,, \\
\label{conv-2}
\lim_{s\to\infty}\, v^{d_s}_{j, \alpha}&= & \,v_{j,\alpha},\quad\forall \alpha\in\N^n;\quad j=1,\ldots,m.
\end{eqnarray}
This in turn implies
\[0\,\preceq\,\M_d(\y)\,\preceq\,\M_d(\y^\B)\,,\quad d\in\N,\]
and for every $j=1,\ldots,m$,
\[\M_d(\v_j)\,\succeq\,0,\quad\M_d((b_j-g_j)\,\v_j)=0,\quad \M_d((b_k-g_k)\,\v_j)\succeq0,\quad k\neq j,\:\:d\in\N,\]
as well as $\M_{d-\ell_j}((t_j-a_j)\,\v_j)\succeq0$ for all $j=1,\ldots,m$ and $d\in\N$.

As $\sup_\alpha\vert\y_\alpha\vert\leq 2^n$, then by Theorem \ref{th:bounded-1}, $\y$ has a representing measure $\mu$ on $[-1,1]^n\cap\{\x: g_1(\x)\leq b_1\}\cdots\cap\{\x: g_m(\x)\leq b_m\}$, that is, on $\om$.
Next,  observe that
\[(n+\vert\alpha\vert) \, L_{\y}(\x^\alpha)   = \displaystyle\sum_{j=1}^m d_j\,b_j\, L_{\v_j}(\x^\alpha ),\quad \forall \alpha\in\N^n,\]
and in particular, 
\[(n+2d) \, L_{\y}(\x^{2d}_i)   = \displaystyle\sum_{j=1}^m d_j\,b_j\, L_{\v_j}(\x_i^{2d}),\quad \forall i=1,\ldots,n;\:d=0,1,\ldots\]
There exists $M>1$ such that $M^{2d}> (n+2d)\,2^n/(\min_jbjd_j)$ for all $d\in\N$ and therefore
\[\sup_{i=1,\ldots,n} L_{\v_j}(x_i^{2d})\,<\,M^{2d}\,=:\,\gamma_d,\quad \forall d=0,1,\ldots\]
As $\M_d(\v_j)\succeq0$ for all $d$, then by \cite[Proposition 3.6]{lasserre2009moments}, $\vert v_{j,\alpha}\vert\leq \gamma_d^{\vert\alpha\vert/2d}=M^{\vert\alpha\vert}$ for all $\vert\alpha\vert\leq 2d$,
and therefore for all $\alpha\in\N^n$. 
In addition as $\M_d((b_k-g_k)\,\v_j)\succeq0$ for all $k\neq j$ and $\M_d((b_j-g_j)\,\v_j)=0$ for all $d$, then by Theorem \ref{th:bounded-1}, $\mu_j$ is supported on $[-M,M]^n\cap\{\x:g_j(\x)=b_j\}\bigcap_{k\neq j}\{\x: g_k(\x)\leq b_k\}$, that is, $\mu_j$ is supported on $\om_j$. Hence 
$(\mu,\mu_1,\ldots,\mu_m)$ is feasible for the infinite-dimensional LP \eqref{lp-1}. On the other hand
\[\rho\, \leq \,\lim_{s\to\infty}\rho_{d_s}\,=\,\lim_{s\to\infty}y^{d_s}_0\,=\,y_0\,=\,\mu(\om),\]
which implies that $(\mu,\mu_1,\ldots,\mu_j)$ is an optimal solution of \eqref{lp-1}. As 
$(\lambda_\om,\phi_1,\ldots,\phi_m)$ is the unique optimal solution of 
\eqref{lp-1}, one must have $\mu=\lambda_\om$ and $\mu_j=\phi_j$ for all $j=1,\ldots,m$.
This in turn implies that all converging subsequences \eqref{conv-1}-\eqref{conv-2} must have the same limit, that is, the whole sequence $(\y^d,\v_1^d,\ldots,\v_m^d)_{d\in\N}$ converges. In other words \eqref{conv-1}-\eqref{conv-2} reads
\begin{eqnarray*}
\lim_{d\to\infty}\,y^{d}_{\alpha}&=&y_\alpha\,=\,\int_\om \x^\alpha\,d\lambda,\quad\forall \alpha\in\N^n \,, \\
\lim_{d\to\infty}\, v^{d}_{j, \alpha}&= &v_{j,\alpha}\,=\,\int_{\om_j}\x^\alpha d\phi_j,\quad\forall \alpha\in\N^n;\quad j=1,\ldots,m,
\end{eqnarray*}
which is the desired result \eqref{th-first-step-1}
\end{proof}

\subsection{Proof of Theorem \ref{th-second-step}}
\label{proof-3.5}
\begin{proof}
We only consider \eqref{second-step} as the proof for \eqref{second-step-inf} is almost a verbatim copy. 
The proof is similar to that of Theorem \ref{th-first-step}.
Let $\u=(u_{\alpha,k})$ be a feasible solution (assumed to exist). 
As $u_{\alpha,0}=v_{j,\alpha}$ for all $\alpha\in\N^n_{2d}$ it follows
that $L_{\u}(x_i^{2d})=L_{\v_j}(x_i^{2d})=\int_{\om_j} x_i^{2d}\,d\phi_j$, and $u_0=\phi_j(\om_j)$,
where $\v_j=(v_{j,\alpha})_{\alpha\in\N^n}$ is the sequence of moments of $\phi_j$, 
already computed in Step 1. 

Next, $\M_{d-1}((\tau-z^2)\,\u)\succeq0$ implies
$L_\u(z^{2d})\leq \tau^d$. Therefore as $\M_d(\u)\succeq0$,
by \cite[Proposition 3.6]{lasserre2009moments}
\[\vert u_{\alpha,k}\vert\,\leq\,
\max[\,\phi_j(\om_j)\,,\,\tau^d\,,\,L_{\v_j}(x_i^{2d})\,],\quad\forall (\alpha,k)\in\N^{n+1}_{2d}.\]
Therefore the feasible set is closed and bounded, hence compact, which implies that \eqref{second-step} has an optimal solution $\u^d$
with optimal value $\overline{\rho}_{j,d}=u^d_{0,1}$.

Next, complete the finite vector $\u^d$ with zeros to make it an infinite sequence $\u^d=(u^d_{\alpha,k})_{(\alpha,k)\in\N^{n+1}}$. 
As in the proof of Theorem \ref{th-first-step} there is a subsequence $(d_s)_{s\in\N}$ and an element $\u=(u_{\alpha,k})_{(\alpha,k)\in\N^{n+1}}$ such that
\begin{equation}
    \label{new-conv}
\lim_{s\to\infty}u^{d_s}_{\alpha,k}\,=\,u_{\alpha,k}\,,\quad \forall (\alpha,k)\in\N^{n+1}.\end{equation}
This in turn implies that $\M_d(\u)\succeq0$, $\M_{d-1}(z\,\u)\succeq0$, $\M_{d-\ell_j}(\theta_j\,\u)=0$, and
$\M_{d-1}((\tau-z^2)\,\u)\succeq0$, for all $d\in\N$. In addition, for all $d\in\N$:
\[L_\u(z^{2d})\leq \tau^d\,;\quad L_\u(x_i^{2d})\leq L_{\v_j}(x_i^{2d})\,=\,\int_{\om_j}x_i^{2d}\,d\phi_j,\quad i=1,\ldots,n.\]
Hence
\[\sum_{d=1}^\infty L_{\u}(z^{2d})^{-1/2d}\,=\,+\infty\,;\quad
\sum_{d=1}^\infty L_{\u}(x_i^{2d})^{-1/2d}\,=\,+\infty,\]
for all $i=1,\ldots,n$. That is, the multivariate Carleman condition \eqref{carleman-1} holds for the sequence $\u$, and by Lemma \ref{carleman}, it has a representing measure $\psi_j$ on $\R^{n+1}$ which is moment determinate. By construction:
\[\int \x^\alpha\,d\psi_j(\x,z)\,=\,\int_{\om_j}\x^\alpha\,d\phi_j,\quad\forall \alpha\in\N^n.\]
As $L_\u(z^{2d})\leq \tau^d$, $\M_d(\u)\succeq0$,  and $\M_d(z\,\u)\succeq0$ for all $d$, then by Theorem \ref{th:bounded-1},
$z\geq0$ and $z^2\leq\tau$ on the support of $\psi_j$. Hence ${\rm supp}(\psi_j)\subset \om_j\times\{z\geq0: z^2\leq \tau\}$. Similarly, as $\psi_j$ has compact support, by \cite[Theorem 3.2(a)]{Las11}, the constraints $\M_d(\theta_j\,\u)=0$ for all $d$,  imply that ${\rm supp}(\psi_j)\subset\{(\x,z): \theta_j(\x,z)=0\}$.
Hence disintegrating $\psi_j$ into its marginal $\phi_j$ on $\om_j$
and its conditional $\hat{\psi}_j(dz\vert\x)$ on $\R_+$ given $\x\in\om_j$, yields:
\begin{eqnarray*}
u_{\alpha,1}\,=\,\int \x^\alpha \,z\,d\psi_j(\x,z)&=&\int_{\om_j}\x^\alpha\,\left(\int_{\R_+}z\,\hat{\psi}_j(dz\vert\x)\right)\,d\phi_j(\x)\\
&=&\int_{\om_j}\x^\alpha \Vert\nabla g_j(\x)\Vert\,d\phi_j(\x)\\
&=&\int_{\om_j}\x^\alpha\,d\sigma_j(\x),\quad\forall\alpha\in\N^n.
\end{eqnarray*}
Hence by \eqref{new-conv} 
\[\lim_{s\to\infty}u^{d_s}_{\alpha,1}=\int_{\om_j}\x^\alpha\,d\sigma_j,\quad\forall\alpha\in\N^n.\]
But the right-hand-side is the same limit for all
converging subsequences and therefore the whole sequence converges, that is, \eqref{th-second-step-1} holds. In particular
$\overline{\rho}_{j,d}=u^d_{0,1}\downarrow u_{0,1}=\sigma_j(\om_j)$ as $d\to\infty$. 

Similarly in \eqref{second-step-inf}, but now  as one minimizes, $\underline{\rho}_{j,d}\uparrow u_{0,1}=\sigma_j(\om_j)$ as $d\to\infty$.
\end{proof}

\subsection{Proof of Theorem \ref{th:signed}}
\label{proof-signed}
\begin{proof}
Let $(\mu_1^{s,\pm},\ldots,\mu_m^{s,\pm})_{s\in\N}$ be a minimizing sequence of \eqref{lp-general-1}, with
\[a_s:=\sum_{i=1}^n\sum_{j=1}^m \left( \int_{\om_j^+} d\mu_j^{s,+}+\int_{\om_j^-} d\mu^{s,-}_j  \right) ,\]
so that $a_s\downarrow\rho$ as $s\to\infty$. 
As $\sup_s\mu^{s,\pm}_j(\om_j^{\pm})<a_0$ it follows that there is a subsequence $(s_t)_{t\in\N}$ and that there are Borel measures
$\mu_j^\pm\in\mathscr{M}_+(\om_j^\pm)$, $j=1,\ldots,m$, such that:
\begin{equation}
\label{conv-3}
\forall j=1,\ldots,m:\quad \lim_{t\to\infty}\,\int_{\om_j^\pm} f\,d\mu^{s_t,\pm}_j\,=\,\int_{\om_j^\pm} f\,d\mu^{\pm}_j,
\end{equation}
for all continuous functions $f$ on $\om_j^\pm$. In particular: with $\alpha\in\N^n$, fixed:
\[(n+\vert\alpha\vert) \,y_\alpha\,=\,\lim_{t\to\infty}
\displaystyle\sum_{j=1}^m \displaystyle\int_{\om_j^\pm} \x^\alpha  \,q_j(\x)\, d\mu^{s_t,\pm}_j\,=\,
\displaystyle\sum_{j=1}^m \displaystyle\int_{\om_j^\pm} \x^\alpha  \, q_j(\x)\,d\mu^{\pm}_j,\]
and $\rho=\sum_{j=1}^m\mu_j^\pm(\om_j)^\pm)$,
which shows that $(\mu_1^{\pm},\ldots,\mu_m^\pm)$ is an optimal solution of the LP \eqref{lp-general-1}.
It remains to prove that $\mu^\pm_j=\phi^\pm_j$ for every $j=1,\ldots,m$.

Observe that $(\phi^\pm_1;\ldots,\phi_m^\pm)$ is also feasible for \eqref{lp-general-1}. Indeed as
$\om$ is compact and $\phi_j(\om_j)<\infty$, then all moments of $\phi_j$ are well-defined
and \eqref{stokes-convex} holds.
Therefore one has
\begin{equation}
    \label{auxx}
  0\,=\,\displaystyle\sum_{j=1}^m \displaystyle\int_{\om_j^+} \x^\alpha  \, q_j(\x)  \,d(\mu^+_j-\phi^+_j)
+\displaystyle\sum_{j=1}^m \displaystyle\int_{\om_j^-} \x^\alpha  \, q_j(\x)  \,d(\mu_j^--\phi^-_j),
\end{equation}
for all $\alpha\in\N^n$.
Next, for every $j=1,\ldots,m$, let $\x\mapsto h_j(\x):=\prod_{k\neq j}(b_k-g_k(\x))$ so that $h_j$ vanishes on $\om_k$ for all $k\neq j$.
Letting $\mu_j:=\mu^+_j+\mu^-_j\in\mathscr{M}_+(\om_j)$ and $\phi_j:=\phi^+_j+\phi^-_j\in\mathscr{M}_+(\om_j)$, and using that for every $\beta\in\N^n$,
\[\x^\beta q_j(\x)\,h_j(\x)\,=\,\sum_\alpha r^\beta_\alpha \x^\alpha,\]
for some $(r^\beta_\alpha)$, then by \eqref{auxx}, 
\begin{eqnarray*}
0&=&\displaystyle\sum_{k=1}^m \displaystyle\int_{\om_k} \x^\beta  \, q_k(\x)\,q_j(\x)\,h_j(\x)  \,d(\mu_k-\phi_k), \quad \beta\in\N^n \\
&=&\displaystyle\int_{\om_j} \x^\beta  \, q_j(\x)^2\,h_j(\x)  \,d(\mu_j-\phi_j), \quad \forall\beta\in\N^n.
\end{eqnarray*}
Therefore let $\Theta_j:=\{\x\in\om_j: h_j(\x)\,q_j(\x)=0\}$. As the above holds for all $\beta\in\N^n$, then under Assumption \ref{ass-exp}, $\mu_j=\phi_j+\eta_j\geq\phi_j$ where $\eta_j\in\mathscr{M}_+(\om_j)$ and ${\rm supp}(\eta_j)\subset\Theta_j$. But then if $\eta_j(\om_j)>0$, 
\begin{eqnarray*}
\rho&=&\sum_{i=1}^n\sum_{j=1}^m \int_{\om_j^+} d\mu_j^++\int_{\om_j^-} d\mu^-_j\\
&=&\sum_{i=1}^n\sum_{j=1}^m \int_{\om_j} d\mu_j\,>\,\sum_{i=1}^n\sum_{j=1}^m \int_{\om_j}d\phi_j\,,
\end{eqnarray*}
which contradicts optimality of $(\mu^\pm_1,\ldots,\mu^\pm_m)$, and so $\mu_j^\pm=\phi_j^\pm$, $j=1,\ldots,m$.
\end{proof}

\subsection{Proof of Theorem \ref{th:general-sdp-1}}
\label{proof-general-sdp-1}
\begin{proof}
Let $(\v_j^{\pm,s})_{s\in\N}$ be a minimizing sequence with
$a\,:=\,\sum_{j=1}^m v^{\pm,0}_{j,0}$ and
so that $\lim_{s\to\infty}\,\sum_{j=1}^m v_{j,0}^{\pm,s}=\rho_d$.
Moreover, let $a\,:=\,\sum_{j=1}^m v^{\pm,0}_{j,0}$ so that
$\sup_jv_{j,0}^{\pm,s}<a$. Moreover the constraints $\M_{d-1}((1-x_i^2)\,\v_j^{\pm,s})\succeq0$ yield
\[L_{\v_j^{\pm,s}}(x_i^{2k})\leq\,v^{\pm,s}_{j,0}, \quad k=1,\ldots,d\,;\:i=1,\ldots,n\,;\:s=0,1,\ldots\]
As $\M_d(\v_j^{\pm,s})\succeq0$, then by \cite[Proposition 3.6]{lasserre2009moments},
\[\vert v_{j,\alpha}^{\pm,s}\vert \,<\, a,\quad \forall \alpha\in\N^n_{2d}\,;\quad j=1,\ldots,m.\]
Therefore  there is a subsequence $(s_k)_{k\in\N}$ and 
$(\v_1^\pm,\ldots,\v_m^\pm)$ such that
\[\lim_{k\to\infty} v_{j,\alpha}^{\pm,s_k}\,=\, v_{j,\alpha}^{\pm},\quad\forall \alpha\in\N^n_{2d};\quad j=1,\ldots,m.\]
This in turn implies that for every $j=1,\ldots,m$:
\begin{equation}
\label{aux44}
\M_d(\v_j^\pm)\,\succeq\,0;\:\M_d((b_j-g_j)\,\v_j^\pm)\,=\,0,
\end{equation}
\begin{equation}
    \label{aux45}
\M_d((b_l-g_l)\,\v_j^\pm)\,\succeq\,0,\:l\neq\,j;
\M_d(q_j\,\v^+_j)\,,\,\M_d(-q_j\,\v^-_j)\succeq0,\end{equation}
and also, $\M_{d-1}((1-x_i^2)\,\v_j^\pm)\succeq0$ for all $i=1,\ldots,n$, which shows that $(\v_j^\pm)$ is a feasible solution. Finally,
\[\rho_d\,=\,\lim_{k\to\infty}\sum_{j=1}^m v_{j,0}^{\pm,s_k} \,=\,
\sum_{j=1}^m v_{j,0}^\pm,\]
implies that $(\v_j^{\pm})_{j\leq m}$ is an optimal solution, and  $\alpha\in\N^n_{2d}$.

So for each $d$, let $(\v^{\pm,d}_j)_{j\leq m}$ be an optimal solution of \eqref{sdp-2} and complete each vector $\v_j^{\pm,d}$ with zeros to make it an infinite sequence $(v_{j,\alpha}^{\pm,d})_{\alpha\in\N^n}$ (still denoted $\v_j^{\pm,d}$ for notational convenience). From what precedes, note that $\vert v_{j,\alpha}^{\pm,d}\vert\leq v_{j,0}^{\pm,d}$ for all $\alpha\in\N^n_{2d}$. Therefore, as \eqref{sdp-2} is a relaxation of \eqref{lp-general-1} for every $d$:
\[\vert v_{j,\alpha}^{\pm,d}\vert\leq v_{j,0}^{\pm,d}\,\leq\,\rho\,,\quad \forall j=1,\ldots,m\,;\:\forall \alpha\in\N^n_{2d},\]
and hence for all $\alpha\in\N^n$, where $\rho$ is the optimal value of \eqref{lp-general-1}.

By a similar argument already used in the proof of Theorem~\ref{proof-3.5}, 
there exists a subsequence $(d_s)_{s\in\N}$ and sequences $\v_j^\pm=(v_{j,\alpha}^\pm)_{\alpha\in\N^n}$, $j=1,\ldots,m$, such that:
\begin{equation}
\label{conv-4}
    \lim_{s\to\infty}v_{j,\alpha}^{\pm,d_s}\,=\,
    v_{j,\alpha}^\pm,\quad \alpha\in\N^n;\:j=1,\ldots,m.
\end{equation}
In particular \eqref{aux44}-\eqref{aux45}  holds for $\v_j^\pm$, and also  $\M_{d-1}((1-x_i^2)\,\v_j^\pm)\succeq0$ for all $i=1,\ldots,n$.
and all $j=1,\ldots,m$. 
For each $j=1,\ldots,m$, the quadratic module
generated by the polynomials 
$\{\pm b_j-g_j,(b_k-g_k)_{k\neq j},q_j,(1-x_i^2)_{i=1,\ldots,n}\}$
(resp. $\{\pm b_j-g_j,(b_k-g_k)_{k\neq j},-q_j,(1-x_i^2)_{i=1,\ldots,n}\}$)
is Archimedean and therefore by Putinar's Positivstellensatz
\cite{Putinar1993positive}, $\v_j^+$ (resp $\v_j^-$) has a representing measure $\mu_j^+$ on $\om_j^+\cap \overline{\B}=\om_j^+$, (resp. $\mu_j^-$ on $\om_j^-\cap \overline{\B}=\om_j^-$), for each $j=1,\ldots,m$.
Moreover, by \eqref{conv-4}:
\begin{eqnarray*}
(n+\vert\alpha\vert) \, y_\alpha  & =&\displaystyle\sum_{j=1}^m L_{\v^\pm_j}(q_j\,\x^\alpha )\\
& =&\displaystyle\sum_{j=1}^m \int_{\om_j^+}\x^\alpha\,q_j\,d\mu_j^+-\int_{\om_j^-}\x^\alpha\,q_j\,d\mu_j^-,
\end{eqnarray*}
for all $\alpha\in\N^n$, and therefore $(\mu_j^\pm)$ is a feasible solution of the LP \eqref{lp-general-1}. 
In addition,
\[\rho\,\geq\,\lim_{k\to\infty}\sum_{j=1}^mv_{j,0}^{\pm,d_k}\,=\,
\sum_{j=1}^mv_{j,0}^{\pm}\,=\,\sum_{j=1}^m\left( \int_{\om_j^+}d\mu_j^++\int_{\om_j^-}d\mu_j^-\right),\]
which shows that $(\v_j^\pm)$ is an optimal solution of the LP \eqref{lp-general-1}. Finally, by Theorem \ref{th:signed},
$\mu_j^\pm=\phi_j^\pm$ for all $j=1,\ldots,m$, 
and \eqref{conv-4} holds for the whole sequence $(\v_j^{\pm,d})_{d\in\N}$,
and yields the desired result \eqref{th:general-sdp-1-1}.
\end{proof}

\bibliographystyle{plain}

\begin{thebibliography}{10}

\bibitem{Ash72RAP}
R.~B. Ash.
\newblock {\em Real Analysis and Probability}.
\newblock Academic Press, New York, 1972.

\bibitem{canny1988complexity}
John Canny.
\newblock {\em The complexity of robot motion planning}.
\newblock MIT press, 1988.

\bibitem{HK14roa}
D.~Henrion and M.~Korda.
\newblock {Convex Computation of the Region of Attraction of Polynomial Control
  Systems}.
\newblock {\em {IEEE Transactions on Automatic Control}}, 59(2):297--312, 2014.

\bibitem{HLS09vol}
D.~Henrion, J.~Lasserre, and C.~Savorgnan.
\newblock {Approximate Volume and Integration for Basic Semialgebraic Sets}.
\newblock {\em SIAM Review}, 51(4):722--743, 2009.

\bibitem{gloptipoly3}
D.~Henrion, J.-B. Lasserre, and J.~L\"ofberg.
\newblock {GloptiPoly 3: moments, optimization and semidefinite programming}.
\newblock {\em Optimization Methods and Software}, 24(4-5):pp. 761--779, August
  2009.

\bibitem{kontsevich2001periods}
M.~Kontsevich and D.~Zagier.
\newblock Periods.
\newblock In {\em Mathematics unlimited-2001 and beyond}, pages 771--808.
  Springer, 2001.

\bibitem{Lairez19}
P.~Lairez, M.~Mezzarobba, and M.~Safey~El Din.
\newblock {Computing the volume of compact sem-algebraic sets}.
\newblock In {\em ISSAC'19: Proceedings of the 2019 {ACM} International
  Symposium on Symbolic and Algebraic Computation}. ACM, Beijing, China, 2019.

\bibitem{Las01sos}
J.-B. Lasserre.
\newblock {Global Optimization with Polynomials and the Problem of Moments}.
\newblock {\em SIAM Journal on Optimization}, 11(3):796--817, 2001.

\bibitem{lasserre2009moments}
J.-B. Lasserre.
\newblock {\em {Moments, Positive Polynomials and Their Applications}}.
\newblock Imperial College Press optimization series. Imperial College Press,
  2009.

\bibitem{Las11}
J.-B. Lasserre.
\newblock A new look at nonnegativity on closed sets and polynomial
  optimization.
\newblock {\em SIAM Journal on Optimization}, 21(3):864--885, 2011.

\bibitem{lasserre-icm}
J.~B. Lasserre.
\newblock The {M}oment-{S}{O}{S} {H}ierarchy.
\newblock In M.~Viana B.~Sirakov, F. Ney de~Sousa, editor, {\em Proc{.}~of the
  International Congress of Mathematicians (ICM 2018)}, volume~4, pages
  3773--3794. World Scientific, Rio de Janiro, Brasil, Aug{.} 2019.

\bibitem{lasserre2019volume}
J.-B. Lasserre.
\newblock {Volume of Sublevel Sets of Homogeneous Polynomials}.
\newblock {\em SIAM Journal on Applied Algebra and Geometry}, 3(2):372--389,
  2019.

\bibitem{invsdp19}
V.~Magron, M.~Forets, and D.~Henrion.
\newblock Semidefinite approximations of invariant measures for polynomial
  systems.
\newblock {\em Discrete \& Continuous Dynamical Systems - B},
  24(1531-3492\_2019\_12\_6745):6745, 2019.

\bibitem{reach19}
V.~Magron, P.-L. Garoche, D.~Henrion, and X.~Thirioux.
\newblock {Semidefinite Approximations of Reachable Sets for Discrete-time
  Polynomial Systems}.
\newblock {\em {SIAM Journal on Control and Optimization}}, 57(4):2799--2820,
  2019.

\bibitem{moseksoft}
{The MOSEK optimization software}.
\newblock \url{http://www.mosek.com/}.

\bibitem{oustry2019inner}
A.~Oustry, M.~Tacchi, and D.~Henrion.
\newblock Inner approximations of the maximal positively invariant set for
  polynomial dynamical systems.
\newblock {\em IEEE Control Systems Letters}, 3(3):733--738, 2019.

\bibitem{Putinar1993positive}
M.~Putinar.
\newblock {Positive polynomials on compact semi-algebraic sets}.
\newblock {\em Indiana University Mathematics Journal}, 42(3):969--984, 1993.

\bibitem{tacchi2019exploiting}
M.~Tacchi, T.~Weisser, J.-B. Lasserre, and D.~Henrion.
\newblock {Exploiting Sparsity for Semi-Algebraic Set Volume Computation}.
\newblock {\em preprint arXiv:1902.02976}, 2019.

\bibitem{taylor}
M.~E Taylor.
\newblock {\em Partial Differential Equations: Basic Theory}.
\newblock Springer-Verlag, New Yor, Inc., Springer Texts in Mathematics, New
  York, 1996.

\bibitem{whitney}
H.~Whitney.
\newblock {\em Geometric Integration Theory}.
\newblock Princeton University Press, Princeton, NJ, 1957.

\end{thebibliography}

\end{document}